\documentclass[12pt,reqno]{amsart}
\usepackage{amsfonts, amsbsy, amsmath, amssymb, mathtools, amsthm,ragged2e}
\usepackage{amscd}
\usepackage{color,enumerate}
\usepackage{breqn}
\usepackage{bm}
\usepackage{pgfplots}
\pgfplotsset{compat=1.15}
\usepackage{tikz}
\usepackage{tikz-cd}
\usepackage{subcaption}
\usetikzlibrary{arrows}
\usepackage{verbatim}
\usepackage{graphicx}
\usetikzlibrary{matrix,calc}
\usepackage{caption}
\usepackage{hyperref}
\usepackage{float}
\usepackage[mathscr]{euscript}
\usepackage{mathrsfs}
\usepackage[margin=2.5cm]{geometry}
\usepackage{epsfig}

\newtheorem{thm}{Theorem}[section]
\newtheorem{lemma}[thm]{Lemma}

\newtheorem{prop}[thm]{Proposition}
\newtheorem{conj}[thm]{Conjecture}
\newtheorem{thm-con}[thm]{Theorem-Conjecture}
\numberwithin{equation}{section}

\theoremstyle{definition}
\newtheorem{defn}[thm]{Definition}
\newtheorem{rmk}[thm]{Remark}

\newtheorem{exmp}[thm]{Example}

\newtheorem{setup}[thm]{Set-up}

\DeclareMathOperator{\supp}{supp}
\DeclareMathOperator{\pol}{pol}

\DeclareMathOperator{\Ass}{Ass}
\DeclareMathOperator{\MinAss}{MinAss}

\newcommand{\x}{\mathbf{x}}

\newcommand{\G}{\mathcal{G}}
\newcommand{\F}{\mathcal{F}}
\newcommand{\la}{\mathcal{L}}

\newcommand{\Cla}{\mathcal{C}}

\newcommand{\D}{\Delta}

\def\l{\left(}
\def\r{\right)}
\def\K{\mathbb{K}}

\begin{document}
\title[The Mendez--Pinto--Villarreal Conjecture]{The Mendez--Pinto--Villarreal Conjecture for some classes of monomial ideals}

\author[Paromita Bordoloi]{Paromita Bordoloi}
\email{2022rma0026@iitjammu.ac.in}
\address{Department of Mathematics, Indian Institute of Technology Jammu, J\&K, India - 181221.}

\author[Kanoy Kumar Das]{Kanoy Kumar Das}
\email{kanoydas0296@gmail.com}
\address{Indian Statistical Institute Kolkata, Stat-Math Unit, 203 B. T. Road, Baranagar, Kolkata 700108, India}

\author[Rajiv Kumar]{Rajiv Kumar}
\email{rajiv.kumar@iitjammu.ac.in}
\address{Department of Mathematics, Indian Institute of Technology Jammu, J\&K, India - 181221.}

\date{\today}


\begin{abstract}
Characterizing when the symbolic and ordinary powers of an ideal coincide is a central problem in commutative algebra, and ideals satisfying this property are called \emph{Simis ideals}. In this article, we investigate the Simis property of monomial ideals by studying the recent conjecture of Mendez, Pinto, and Villarreal on monomial ideals with minimal irreducible decomposition. Let $I$ be a monomial ideal, and let $\mathcal{F}(I)$ denote the collection of supports of the minimal generators of $I$. Assuming that $\mathcal{F}(I)=\mathcal{F}(\sqrt{I})$, we prove that if $I$ admits more than one minimal generator with a fixed support, then it is not Simis. Using this reduction, we establish the Mendez--Pinto--Villarreal conjecture for two broad classes of monomial ideals, namely support-$3$ monomial ideals and monomial ideals whose associated simplicial complexes are simplicial forests. Finally, we study the Cohen--Macaulay property of monomial ideals whose associated simplicial complexes are grafted and satisfy $\mathcal{F}(I)=\mathcal{F}(\sqrt{I})$.
\end{abstract}

\maketitle

\section{Introduction and Preliminaries}\label{Intro_Pre}

Let $S=\K[x_1,\ldots , x_n]$ be a polynomial ring over a field $\K$, and let $I\subseteq S$ be an ideal. For any integer $s \geq 1$, the $s$-th symbolic power of $I$, denoted by $I^{(s)}$, is defined as
\[
I^{(s)} = \bigcap_{P \in \MinAss(I)} (I^s S_P \cap S),
\]
where $\MinAss(I)$ denotes the set of minimal primes of $I$. Ideals for which the equality $I^{(s)}=I^s$ holds for all $s\geq 1$ are known as \emph{Simis ideals}, named in honor of Aron Simis for his foundational contributions to the study of symbolic powers \cite{mendez2024symbolic, SimisRamos2022, SimisTrung1988, SimisUlrich2000,  SimisVasconcelosVillarreal1994}. The problem of characterizing Simis ideals has attracted considerable interest, yet a complete classification, even for restricted classes of ideals remains highly nontrivial.

In this article, we study the Simis property for monomial ideals. For square-free monomial ideals, the characterization of Simis ideals is intimately connected to the packing property in combinatorial optimization through the Conforti-Cornu\'ejols conjecture \cite{Packing, GRV2009}. A classical result of Simis, Vasconcelos, and Villarreal \cite{SimisVasconcelosVillarreal1994} provides such a characterization for the edge ideals of graphs. Since then, the Simis property has been investigated across a variety of other classes of monomial ideals; these include edge ideals of weighted oriented graphs \cite{GMV2024, MandalPradhan2021}, certain generalized edge ideals \cite{das2024equality}, cover ideals of graphs \cite{GRV2005, GRV2009}, $3$-path ideals of graphs \cite{alilooee2021packing}, matroidal ideals \cite{ficarra2025symbolic}, and dual of $t$-connected ideals \cite{BhardwajDasSawant2026}, among others.  For monomial ideals that are not necessarily square-free, embedded associated primes may occur, and naturally this leads to an alternative definition of symbolic powers involving the full set of associated primes. This formulation of symbolic powers has also been studied extensively in the literature. We refer the interested reader to \cite{article, dao2015symbolic, HochsterHuneke2002} and the references therein for further background and developments.

When $I$ is a monomial ideal, symbolic powers admit a simpler description in terms of the primary components of $I$. Let $I=\bigcap_{P\in \Ass(I)}Q(P)$ be the minimal primary decomposition of $I$, where $Q(P)$ denotes the $P$-primary component  of $I$ corresponding to $P\in \Ass(I)$. Then
\[
I^{(s)}=\bigcap_{P\in \MinAss(I)}Q(P)^s,
\]
see \cite[Lemma 2]{gimenez2018symbolic}. An ideal $Q\subseteq S$ is called \emph{irreducible} if it cannot be expressed as an intersection of two ideals properly containing it. It follows that, after a suitable permutation of the variables, every irreducible monomial ideal is of the form $\left( x_i^{a_i}\mid i\in [r],\ r\leq n,\ a_i\geq 1\right)$. In particular, irreducible monomial ideals are primary, and every monomial ideal admits a unique irredundant irreducible decomposition. Although the irreducible decomposition of a monomial ideal need not be minimal in general, there are important classes of monomial ideals for which the irreducible decomposition is minimal, including square-free monomial ideals and edge ideals of weighted oriented graphs \cite{PRT2019}. In this article, we investigate the Simis property of monomial ideals with minimal irreducible decomposition through a recent conjecture of Mendez, Pinto, and Villarreal.

\begin{conj}\cite[Conjecture 5.7]{mendez2024symbolic}\label{conj}
    Let $I$ be a monomial ideal without any embedded associated primes. If the irreducible decomposition of $I$ is minimal, and $I$ is a Simis ideal, then there is a Simis square-free monomial ideal $J$ and a standard linear weighting $w$ such that $I=J_w$.   
\end{conj}
In a recent article, authors verified this conjecture for support-$2$ monomial ideals \cite[Theorem 2.3]{Bordoloi2026}. The conjecture is also known to hold for certain classes of monomial ideals obtained by imposing suitable conditions on the generating sets of their associated primes \cite[Theorem 4.7]{kumar2025waldschmidt}. In this article, we further investigate the conjecture for the classes of support-$3$ monomial ideals and monomial ideals associated to simplicial forests.

We now recall the necessary definitions and fix the notation that is used throughout the article. Let $I\subseteq S$ be a monomial ideal. For simplicity, we write a monomial in $S$ as $\x^{\mathbf{a}}=x_1^{a_1}\cdots x_n^{a_n}$, where $\mathbf{a}=(a_1,\ldots,a_n)\in \mathbb{N}^n$. The \emph{support} of the monomial $\x^{\mathbf{a}}$ is defined by $\supp(\x^{\mathbf{a}})\coloneqq \{x_i\mid a_i\neq 0\}$. Let $\G(I)$ denote the unique minimal monomial generating set of $I$. We define
\[
\mathcal{F}(I)\coloneqq \{\supp(\x^{\mathbf{a}})\mid \x^{\mathbf{a}}\in \G(I)\}.
\]
Throughout this article, we assume that $\mathcal{F}(I)=\mathcal{F}(\sqrt{I})$. Equivalently, there is no proper containment among the supports of the minimal generators of $I$. The interaction between combinatorics and commutative algebra has led to substantial developments in the study of monomial ideals. In particular, combinatorial objects such as graphs, hypergraphs, and simplicial complexes often encode important algebraic properties of square-free monomial ideals. To this end, the collection $\mathcal{F}(I)$ can naturally be viewed as the facet set of a simplicial complex on the vertex set $[n]=\{1,\ldots,n\}$. We denote this simplicial complex by $\Delta(I)$ and write $\mathcal{F}(\Delta(I))$ for the set of facets. Thus, $\mathcal{F}(\Delta(I))=\mathcal{F}(I).$ We call $\Delta(I)$ the simplicial complex associated to $I$.

A convenient framework for studying monomial ideals is provided by weighted monomial ideals (see \cite[Section 4]{mendez2024symbolic}). In this setting, for a monomial $\x^{\mathbf{a}}=x_1^{a_1}\cdots x_n^{a_n}\in S$, the exponent $a_i$ of the variable $x_i$ is called the \emph{weight} of $x_i$ in $\x^{\mathbf{a}}$. A variable $x_i$ is said to have \emph{uniform weight} (with respect to $I$) if all minimal generators of $I$ whose supports contain $x_i$ have the same weight. We say that $I$ has a \emph{uniform weighting}, or equivalently a \emph{standard linear weighting} (see \cite{mendez2024symbolic}), if every variable has uniform weight with respect to $I$.

Let $I\subseteq S$ be a monomial ideal satisfying $\mathcal{F}(I)=\mathcal{F}(\sqrt{I})$. Write $\mathcal{G}(I)=\{\x^{\mathbf{a}_1},\ldots,\x^{\mathbf{a}_m}\}$, where $\mathbf{a}_i=(a_{i,1},\ldots,a_{i,n})\in \mathbb{N}^n$ for $1\leq i\leq m$. For each $F\in \mathcal{F}(I)$, let
\[
\mathcal{G}(I,F)\coloneqq \{\x^{\mathbf{a}}\in \mathcal{G}(I)\mid \supp(\x^{\mathbf{a}})=F\}.
\]

\begin{enumerate}
    \item Define $\alpha_I(F)\coloneqq |\mathcal{G}(I,F)|$. Thus, $\alpha_I(F)$ denotes the number of minimal monomial generators of $I$ whose support is $F$.
    
    \item Let $F\in \mathcal{F}(I)$ and suppose that $\mathcal{G}(I,F)=\{\x^{\mathbf{a}_1},\ldots,\x^{\mathbf{a}_r}\}$, where $\mathbf{a}_t=(a_{t,1},\ldots,a_{t,n})\in \mathbb{N}^n$ for $1\leq t\leq r$. For $x_i\in F$, we define $w(x_i,F,t)\coloneqq a_{t,i}$, namely, the exponent of $x_i$ in the monomial $\x^{\mathbf{a}_t}$. In the case $\alpha_I(F)=1$, we simply write $w(x_i,F)$ instead of ~$w(x_i,F,1)$.
    
    \item For $x_i\in F$, let $\mu_i(F)$ denote the maximum exponent of $x_i$ among the monomials in ~$\mathcal{G}(I,F)$.
    
    \item For $x_i\in F$, let $\nu_i(F)$ denote the minimum exponent of $x_i$ among the monomials in ~$\mathcal{G}(I,F)$.
\end{enumerate}

\medskip
We now summarize the main results of the article. We begin our study of the Mendez--Pinto--Villarreal conjecture (Conjecture~\ref{conj}) by establishing a key reduction. In Lemma~\ref{lem:Support_gen}, we show that, to study the conjecture for monomial ideals $I$ satisfying $\F(I)=\F(\sqrt{I})$, it suffices to consider the case where $\alpha_I(F)=1$ for every $F\in\F(I)$. This reduction plays a crucial role in establishing the conjecture for two important classes of monomial ideals.

\medskip
\noindent
\textbf{Theorems~\ref{thm:simpl_tree_conj} and \ref{support-3}.}
Let $I\subseteq S$ be a monomial ideal. Then Conjecture~\ref{conj} holds in each of the following cases:
\begin{enumerate}
    \item $\mathcal{F}(I)=\mathcal{F}(\sqrt{I})$ and $\Delta(I)$ is a simplicial forest;
    
    \item $I$ is a support-$3$ monomial ideal.

\end{enumerate}

The proofs of the above results rely on several technical lemmas. In the case of support-$3$ monomial ideals, in order to establish the existence of a standard linear weighting, we analyze two separate situations. The first case, treated in Lemma~\ref{Lem:Support3Int_2}, considers the situation where two distinct supports intersect in two vertices. The second case, addressed in Lemma~\ref{lem:supp3_int1}, deals with the case where the intersection consists of exactly one vertex.

\medskip
Finally, we study Cohen--Macaulay property of monomial ideals $I$ whose associated simplicial complex $\Delta(I)$ is a grafted simplicial complex (see Definition~\ref{Grafting_Defn}).

\medskip
\noindent \textbf{Theorem~\ref{Grafting_CM}.}
Let $I \subseteq S=\K[x_1, \ldots, x_m,y_1, \ldots,y_M]$ be a monomial ideal satisfying $\mathcal{F}(I)=\mathcal{F}(\sqrt{I})$, and suppose that $\D(I)$ is a grafting of a simplicial complex $\D'$. Let $J = I \cap \K[x_1, \ldots, x_m]$, and assume that $J$ admits a standard linear weighting. Then $S/I$ is Cohen--Macaulay if and only if the following conditions hold: 
    \begin{enumerate}[\rm(i)]
        \item $\alpha_I(F)=1$ for every $F\in \la(\D(I))$;
        \item for every $F\in \la(\D(I))$, $G\in \D'$, and every vertex $x_i\in F\cap G$, $w(x_i,F)\ge w(x_i,G)$.
    \end{enumerate}

\medskip
The remainder of the paper is organized as follows. In Section~\ref{sec: MPV conj for Simplicial trees}, we establish Conjecture~\ref{conj} for monomial ideals whose associated simplicial complexes are simplicial forests and investigate the Cohen--Macaulay property of monomial ideals whose associated simplicial complexes are grafted. Section~\ref{sec:supp3} is devoted to proving Conjecture~\ref{conj} for support-$3$ monomial ideals.

\section*{Acknowledgements}
The authors acknowledge the use of the computer algebra system Macaulay2 \cite{M2} and the online platform SageMath for carrying out several computations.

\section{Monomial Ideals Associated to Simplicial Forests}\label{sec: MPV conj for Simplicial trees}

In this section, we establish the Mendez--Pinto--Villarreal Conjecture for monomial ideals associated to simplicial forests. Throughout this section, let $I\subseteq S$ be a monomial ideal satisfying $\mathcal{F}(I)=\mathcal{F}(\sqrt{I})$. Although this condition excludes proper containment among supports of minimal generators, distinct generators may still share the same support. Indeed, Simis monomial ideals with this property exist (see \cite[Example~4.7]{Bordoloi2026}). The first key result of this section shows that, under the assumptions of the Conjecture, each support can correspond to at most one minimal monomial generator. This reduction is fundamental to the proofs of the subsequent results in this article.

\begin{lemma}\label{lem:Support_gen}
    Let $I\subseteq S$ be a monomial ideal such that $\mathcal{F}(I)=\mathcal{F}(\sqrt{I})$ and the irreducible decomposition of $I$ is minimal. If $\alpha_I(F) \ge 2$ for some $F \in \F(I)$, then $I^{(2)} \ne I^2$.
\end{lemma}
\begin{proof}
    We construct a monomial $f\in I^{(2)}\setminus I^2$. Without loss of generality, assume that $F = \{x_1,\ldots,x_d\}$. Write $\mathcal{G}(I, F)=\{\x^{\mathbf{b}_i}\mid 1\leq i\leq \alpha_I(F)\}$, where
    \[
    \deg(\x^{\mathbf{b}_1}) \le \deg(\x^{\mathbf{b}_2}) \le \cdots \le \deg(\x^{\mathbf{b}_{\alpha_I(F)}}).
    \]
    Since $\x^{\mathbf{b}_1},\x^{\mathbf{b}_2} \in \G(I,F)$ are distinct, there exist $1 \le p<q \le d$ such that $b_{1,p} > b_{2,p}$ and $b_{1,q} < b_{2,q}$. For simplicity, we assume that $p = 1$ and $q =2$. For each $1 \le j\le d$, we set
    \begin{align*}
        m_j &= \min\left\{b_{1,j},b_{2,j}\right\}, \\
        M_j &= \max\left\{b_{1,j},b_{2,j}\right\}.
    \end{align*}
    Consider the monomial $f = \prod_{j = 1}^d x_j^{\max\left\{2m_j,M_j\right\}}$. We show that $f \in I^{(2)} \setminus I^2$. We begin by proving that $f\notin I^2$. Let $J=I\cap \K[x_1, \dots  x_d]$. By our assumption that there is no proper containment among the supports of the minimal generators of the ideal $I$, we have $\mathcal{G}(J)=\mathcal{G}(I, F)$. Therefore, it suffices to prove that $f\notin J^2$. We first verify that $f\notin \left( \x^{\mathbf{b}_1}, \x^{\mathbf{b}_2}\right)^2$ by considering the following cases.

    \smallskip
    \noindent \textsc{Case 1}. Assume that $\max\{b_{1,1},2b_{2,1}\} = b_{1,1}$ and $\max\{b_{2,2},2b_{1,2}\} = b_{2,2}$. Since $b_{1,1}\geq 1$, we have $b_{1,1}<2b_{1,1}$. This implies that $\x^{2\mathbf{b}_1} \nmid f$. Similarly, since $1\leq b_{2,2} < 2b_{2,2}$, we have $\x^{2\mathbf{b}_2} \nmid f$. Moreover, $b_{1,1} < b_{1,1} + b_{1,2}$ implies that $\x^{\mathbf{b}_1+\mathbf{b}_2} \nmid f$.

    \smallskip
    \noindent \textsc{Case 2}. Assume that $\max\{b_{1,1},2b_{2,1}\} = 2b_{2,1}$ and $\max\{b_{2,2},2b_{1,2}\} = 2b_{1,2}$.  Since $2b_{2,1} < 2b_{1,1}, ~ 2b_{1,2} < 2b_{2,2}$, and $2b_{2,1} < b_{1,1} + b_{2,1}$, we have $\x^{2\mathbf{b}_1} \nmid f$, and $\x^{2\mathbf{b}_2} \nmid f$, and $\x^{\mathbf{b}_1+\mathbf{b}_2} \nmid f$. 

    \smallskip
    \noindent \textsc{Case 3}. Assume that $\max\{b_{1,1},2b_{2,1}\} = b_{1,1}$ and $\max\{b_{2,2},2b_{1,2}\} = 2b_{1,2}$. Since $b_{1,1} < 2b_{1,1},~ 2b_{1,2} < 2b_{2,2}$, and $b_{1,1} < b_{1,1}+b_{2,1}$, we have $\x^{2\mathbf{b}_1} \nmid f$, and $\x^{2\mathbf{b}_2} \nmid f$, and $\x^{\mathbf{b}_1+\mathbf{b}_2} \nmid f$. 

    \smallskip
    \noindent \textsc{Case 4}. Assume that $\max\{b_{1,1},2b_{2,1}\} = 2b_{2,1}$ and $\max\{b_{2,2},2b_{1,2}\} = b_{2,2}$. Since $2b_{2,1} < 2b_{1,1},~ b_{2,2} < 2b_{2,2}$, and $b_{2,2} < b_{1,2}+b_{2,2}$, we have $\x^{2\mathbf{b}_1} \nmid f$, and $\x^{2\mathbf{b}_2} \nmid f$, and $\x^{\mathbf{b}_1+\mathbf{b}_2} \nmid f$. 

    \smallskip
    We now show that $f \notin J^2$. Suppose, to the contrary, that $f \in J^2$. Then, there exist $\x^{\mathbf{a}}, \x^{\mathbf{c}} \in \G(J)$  such that $\x^{\mathbf{a}+\mathbf{c}} \mid f$. Since $f\notin(\mathbf{x}^{\mathbf{b}_1},\mathbf{x}^{\mathbf{b}_2})^2$, we may assume that $\deg(\x^{\mathbf{a}}) \ge  \deg(\x^{\mathbf{b}_1})$ and $\deg(\x^{\mathbf{c}}) \ge  \deg(\x^{\mathbf{b}_2})$. Hence,
    \begin{equation}\label{eqn: 201}
        \deg(\x^{\mathbf{a}}) + \deg(\x^{\mathbf{c}}) \ge  \deg(\x^{\mathbf{b}_1}) + \deg(\x^{\mathbf{b}_2}).
    \end{equation}   
    Since $\x^{\mathbf{a}+\mathbf{c}} \mid f$, we have $a_j + c_j \le \max\{2m_j,M_j\}\le  b_{1,j} + b_{2,j}$ for all $1 \le j \le d$. Moreover, if $\max\{ b_{1,1},2b_{2,1}\} = b_{1,1}$, then $a_1 + c_1 \le b_{1,1} < b_{1,1} + b_{2,1}$. On the other hand, if $\max\{ b_{1,1}, 2b_{2,1}\} = 2b_{2,1}$, then $a_1 + c_1 \le 2b_{2,1} < b_{1,1} + b_{2,1}$. Therefore, we have 
    \begin{equation}\label{eqn: 202}
        \deg(\x^{\mathbf{a}}) + \deg(\x^{\mathbf{c}}) <  \deg(\x^{\mathbf{b}_1}) + \deg(\x^{\mathbf{b}_2}).
    \end{equation}
    This contradicts \eqref{eqn: 201}. Therefore, $f \notin J^2$. It remains to show that $f\in I^{(2)}$. Let $P \in \MinAss(I)$, and let $Q(P)$ denote the corresponding primary component. It suffices to prove that $f \in Q(P)^2$. After relabeling the variables, if necessary, we assume that $x_1, \ldots, x_k \in P$ and $x_{k+1}, \ldots, x_d \notin P$ for some $k \ge 1$. First, suppose that $k = 1$. Then $x_1^{m_1} \in Q(P)$ and hence $x_1^{2m_1} \in Q(P)^2$. Since $x_1^{2m_1} \mid f$, it follows that $f \in Q(P)^2$. Now assume that $k \ge 2$. For each $1 \le i \le k$ there exists a positive integer $w_i$ such that $x_i^{w_i} \in Q(P)$. Since $Q(P)$ is irreducible and $\x^{\mathbf{b}_1}, \x^{\mathbf{b}_2}\in Q(P)$, there exists $x_j, x_{\ell}\in P, 1 \le j , \ell \le k$ such that $w_j \le b_{1,j}$ and $w_{\ell} \le b_{2,\ell}$. If $j=\ell$, then we have $w_j\leq \min\{b_{1,j}, b_{2,j}\}=m_j$, and hence $x_j^{2w_j}\mid f$. This implies that $f\in Q(P)^2$. On the other hand, if $j\neq \ell$, then  $x_j^{b_{1,j}}\mid f, x_{\ell}^{b_{2,\ell}} \mid f$, and hence $x_j^{w_j}x_{\ell}^{w_{\ell}}\mid f$. Hence, $f\in Q(P)^2$, and this completes the proof.
\end{proof}

The preceding lemma allows us to reduce to the case where $\alpha_I(F)=1$ for every $F\in\mathcal{F}(I)$. Accordingly, throughout the remainder of this article, we assume that this condition holds.

\medskip
Let $\D$ be a simplicial complex, and let $\mathcal{F}(\D)$ denote the set of all facets of $\D$. We write 
\[
\D=\left \langle F_1,\dots , F_t \right \rangle.
\]
whenever $\F(\D)=\{F_1,\dots , F_t\}$. A facet $F\in \F(\D)$ is called a \emph{leaf} if either $F$ is the only facet of $\D$, or there exists a facet $G\in \F(\D)$ with $G\neq F$ such that $H\cap F\subseteq G\cap F$ for every facet $H\in \F(\D)$ with $H\neq F$. Any such facet $G$ is called a \emph{joint} of $F$.

A \emph{subcomplex} of  $\D$ is a simplicial complex $\D'$ satisfying $\F(\D')\subseteq \F(\D)$. A connected simplicial complex $\D$ is called a \emph{simplicial tree} if every non-empty subcomplex of $\D$ has a leaf. A \emph{simplicial forest} is a simplicial complex whose connected components are simplicial trees. Thus, simplicial trees and simplicial forests naturally generalize the notions of trees and forests in graph theory to higher dimensions.

The facet ideals of simplicial forests constitute an important class of Simis ideals. Herzog, Hibi, Trung, and Zheng proved that simplicial forests contain no special cycle of length at least three \cite[Theorem~3.2]{HHTZ2008}. Consequently, every simplicial forest is Mengerian, and hence its facet ideal is Simis \cite[Corollary~1.6]{HHTZ2008}.

We now establish the conjecture for monomial ideals whose underlying simplicial complex is a simplicial forest. The proof relies on identifying sufficient conditions under which such ideals admit a standard linear weighting. To this end, we first recall the notion of an induced subcomplex.

Let $\D$ be a simplicial complex and let $T\subseteq V(\D)$. The \emph{induced subcomplex} of $\D$ on the vertex set $T$ is the subcomplex
\[
    \D[T] = \left\langle H\in\F(\D)\mid H\subseteq T \right\rangle .
\]

\begin{prop}\label{Prop_Sim_tree}
    Let $I$ be a monomial ideal, and let $\D(I)$ denote the associated simplicial complex. Assume that 
    \begin{enumerate}
        \item $\mathcal{F}(I)=\mathcal{F}(\sqrt{I})$;
        \item the irreducible decomposition of $I$ is minimal;
        \item $I$ is a Simis ideal.
    \end{enumerate}
    Then for any $F, G \in \F(\D(I))$ satisfying  $\D(I)\left[F\cup G\right] = \left\langle F, G\right\rangle$, we have  $w(x, F) = w(x,G)$ for all $x \in F \cap G$.
\end{prop}
\begin{proof}
    By Lemma~\ref{lem:Support_gen}, we may assume that $\alpha_I(F)=1$ for every $F\in \mathcal{F}(I)$. Let $\x^{\mathbf{a}_1}, \x^{\mathbf{a}_2}\in \G(I)$ be minimal monomial generators with supports $F = \{x_1, \ldots,x_k,x_{k+1} \ldots, x_p\}$ and $G= \{x_1, \ldots, x_k,x_{p+1},\ldots,x_q\}$, respectively, where $1\leq  k < p < q \le n$. Suppose, to the contrary, that $w(x_i, F) \ne w(x_i, G)$ for some $1 \le i \le k$. Without loss of generality, we may assume that $i = 1$. For each $1 \le i \le k$, set
    \begin{align*}
        m_i &= \min\left\{w(x_i,F),w(x_i,G)\right\}, \\
        M_i &= \max\left\{w(x_i,F),w(x_i,G)\right\}.
    \end{align*}   
    Consider the monomial
    \[
    f = \prod_{i = 1}^{k}x_{i}^{\max\left\{2m_i,M_i\right\}}\prod_{i = k+1}^{p}x_{i}^{w(x_i,F)} \prod_{i = p+1}^{q}x_i^{w(x_i,G)}.
    \]
    We show that $f \in I^{(2)} \setminus I^2$. We begin by proving that $f \notin I^2$. Since $x_{k+1}^{2w(x_{k+1}, F)} \nmid f$ and $x_{p+1}^{2w(x_{p+1}, G)} \nmid f$, it follows that $\x^{2\mathbf{a}_1} \nmid f$ and $\x^{2\mathbf{a}_2} \nmid f$. Moreover, since $w(x_1, F) + w(x_1,G) > \max\{2m_1,M_1\}$, and hence $\x^{\mathbf{a}_1+\mathbf{a}_2} \nmid f$. Since $\D(I)\left[F \cup G\right] = \left\langle F, G\right\rangle$, it follows that $\mathbf{x}^{\mathbf{a}_1}$ and $\mathbf{x}^{\mathbf{a}_2}$ are the only minimal monomial generators of $I$ whose supports are contained in $F\cup G$. Therefore, $f\notin I^2$. It remains to show that $f\in I^{(2)}$. Let $P\in \MinAss(I)$, and let $Q(P)$ denote the corresponding primary component. It suffices to prove that $f\in Q(P)^2$. Since $\x^{\mathbf{a}_1}, \x^{\mathbf{a}_2} \in Q(P)$ and the irreducible decomposition of $I$ is minimal, there exists $x_i\in F$ and $x_j\in G$ such that $x_i^{w(x_i,F)}\in Q(P)$ and $x_j^{w(x_j,G)}\in Q(P)$. If $i=j$, then $x_i^{2m_i}\in Q(P)^2$. Since $x_i^{2m_i}\mid f$, it follows that $f\in Q(P)^2$. Now, suppose that $i\neq j$. Then $x_i^{w(x_i,F)}x_j^{w(x_j,G)}\in Q(P)^2$. Since $x_i^{w(x_i,F)}x_j^{w(x_j,G)}\mid f$, we again conclude that $f\in Q(P)^2$. Hence, $f\in Q(P)^2$ for every $P\in \MinAss(I)$, and consequently $f\in I^{(2)}$. Thus, $I^{(2)}\neq I^2$, contradicting the assumption that $I$ is Simis. Therefore, $w(x_i,F)=w(x_i,G)$ for all $1\leq i\leq k$. This completes the ~proof.
\end{proof}

The preceding results allow us to establish one of the main theorems of this article.\begin{thm}\label{prop:simp_tree_conj}\label{thm:simpl_tree_conj}
    Let $I$ be a monomial ideal, and let $\D(I)$ denote the associated simplicial complex. Assume that
    \begin{enumerate}
        \item $\mathcal{F}(I)=\mathcal{F}(\sqrt{I})$; 
        \item the irreducible decomposition of $I$ is minimal;
        \item $I$ is a Simis ideal.
    \end{enumerate}
    If $\D(I)$ is a simplicial tree, then $I$ has a standard linear weighting. In particular, Conjecture~\ref{conj} holds for $I$.
\end{thm}
\begin{proof}
    By Lemma~\ref{lem:Support_gen}, we may assume that $\alpha_I(F)=1$ for all $F\in \mathcal{F}(I)$. Let $F_1,F_2\in\F(\D(I))$ be such that $F_1\cap F_2\neq \emptyset$. We first observe that the induced subcomplex $\D(I)[F_1 \cup F_2]$ is a simplicial tree. Since every simplicial tree admits a good leaf \cite[Corollary 3.4]{MR2434285}, it follows that $\D(I)[F_1 \cup F_2]$ has a good leaf. Note that any facet $H\in\D(I)[F_1\cup F_2]$ distinct from $F_1$ and $F_2$ has no free vertex, since every vertex of $H$ belongs to either $F_1$ or $F_2$. Hence, $H$ cannot be a leaf of $\D(I)[F_1 \cup F_2]$. It follows that either $F_1$ or $F_2$ is a good leaf of $\D(I)[F_1 \cup F_2]$. Without loss of generality, assume that $F_1$ is a good leaf. Then by \cite[Theorem 3.4]{faridi2014good}, there exists an ordering $F_1, G_1, \dots, G_r$ of the facets of $\D(I)[F_1 \cup F_2]$ such that
    \[
    F_1 \cap G_1 \supseteq F_1 \cap G_2 \supseteq \cdots \supseteq F_1 \cap G_r,
    \]
    and, for each $1 \le i \le r$, the facet $G_i$ is a leaf of the subcomplex $\left\langle F_1, G_1, \ldots, G_i \right\rangle$. We show that, for each $1 \le i \le r$ every vertex in $F_1 \cap G_i$ has the same weight with respect to $F_1$ and $G_i$. Since $F_2 = G_j$ for some $1 \le j \le r$, this will imply that every vertex in $F_1 \cap F_2$ has the same weight with respect to $F_1$ and $F_2$. We proceed by induction on $r$. We first consider the case $r = 1$. Suppose, to the contrary, assume that not every vertex in $F_1 \cap G_1$ has the same weight with respect to $F_1$ and $G_1$. Then there exists $x_\ell \in F_1 \cap G_1$ such that $w(x_\ell, F_1) \ne w(x_\ell, G_1)$. Since $\D(I)[F_1 \cup G_1] = \left\langle F_1, G_1 \right\rangle$, it follows from Proposition \ref{Prop_Sim_tree} that $I^{(2)} \ne I^2$, contradicting the assumption that $I$ is Simis. Hence, every vertex in $F_1 \cap G_1$ has the same weight with respect to $F_1$ and $G_1$. Now assume that $r \ge 2$ and the assertion holds for $r-1$. That is, for the ordering $F_1, G_1, \dots, G_{r-1}$ satisfying 
    \[
    F_1 \cap G_1 \supseteq F_1 \cap G_2 \supseteq \cdots \supseteq F_1 \cap G_{r-1},
    \]
    every vertex in $F_1 \cap G_i$ has the same weight with respect to $F_1$ and $G_i$ for each $1\leq i\leq r-1$. It remains to prove that every vertex in $F_1 \cap G_r$ has uniform weight. Suppose, to the contrary, that this is not the case. Write $F_1 = \{x_1, \dots, x_k,x_{k+1}, \dots, x_p\}$ and $G_r = \{x_1, \dots, x_k, x_{p+1}, \dots, x_q\}$, where $k < p < q \le n$.
    Then, $F_1 \cap G_r = \{x_1, \dots, x_k\}$. Assume that there exists $x_\ell \in F_1 \cap G_r$ such that $w(x_\ell,F_1) \neq w(x_\ell,G_r)$. Without loss of generality, assume that $\ell = 1$. We shall prove that $I^{(2)} \neq I^2$. For each $1 \le \ell \le k$, set
    \[
        m_\ell = \min\left\{w(x_\ell,F_1),w(x_\ell,G_r)\right\} ~ \text{and} ~ M_\ell = \max\left\{w(x_\ell,F_1),w(x_\ell,G_r)\right\}.
    \]
    Now, consider the monomial
    \[
    f = \prod_{\ell = 1}^{k}x_{\ell}^{\max\left\{2m_\ell,M_\ell\right\}}\prod_{\ell = k+1}^{p}x_{\ell}^{w(x_\ell,F_1)} \prod_{\ell = p+1}^{q}x_\ell^{w(\ell,G_r)}.
    \]    
    We show that $f \in I^{(2)} \setminus I^2$. Since the irreducible decomposition of $I$ is minimal, the same argument as in the proof of Proposition \ref{Prop_Sim_tree} shows that $f \in I^{(2)}$. It now remains to show that $f \notin I^2$. For each facet $F_1, G_1, \dots, G_r$ of $\D(I)$, let $\x^{\mathbf{a}_1} \in \G(I), \x^{\mathbf{b}_i} \in \G(I), 1\leq i\leq r$ be the monomial generators of $I$ such that $\supp(\x^{\mathbf{a}_1}) = F_1$ and $ \supp(\x^{\mathbf{b}_i}) = G_i,  1\leq i\leq r$. Let $J=\l \x^{\mathbf{a}_1} , \x^{\mathbf{b}_i}\mid 1\leq i\leq r\r$. It suffices to show that $f\notin J^2$. Let $s$ be the smallest index such that $F_1 \cap G_s = F_1 \cap G_r$. Then
    \[
    F_1 \cap G_1 \supseteq \cdots \supseteq F_1 \cap G_{s-1}\supseteq F_1 \cap G_{s}=F_1 \cap G_{s+1}=\cdots =F_1 \cap G_{r},
    \]
    and $F_1 \cap G_i \supsetneq F_1 \cap G_s$ for each $1 \le i \le s-1$. Now, we consider the following cases:

    \smallskip \noindent
    \textsc{Case 1}. Since the exponent of $x_{k+1}$ in $f$ is $w(x_{k+1},F_1)$, we have  $x_{k+1}^{2w(x_{k+1},F_1)} \nmid f$. Hence $\mathbf{x}^{2\mathbf{a}_1} \nmid f$.

    \smallskip \noindent
    \textsc{Case 2}. We shall show that $\x^{\mathbf{a}_1+\mathbf{b}_{i}} \nmid f$ for any $1\leq i\leq r$. If $i=r$, then $w(x_1,F_1) + w(x_1,G_r) > \max\{2m_1, M_1\}$, and therefore $\mathbf{x}^{\mathbf{a}_1+\mathbf{b}_r} \nmid f$. Now suppose that $s \le i \le r-1$. Take any $x_j\in  G_i \setminus G_r$. Since $F_1\cap G_i=F_1\cap G_r$, we have $x_j\notin F_1$. Hence, $x_j\notin F_1\cup G_r= \supp(f)$, which implies that $\x^{\mathbf{a}_1+\mathbf{b}_{i}} \nmid f$. Finally, suppose that $1 \le i \le s-1$. Since $F_1\cap G_i\neq F_1\cap G_r=\{x_1, \dots , x_k\}$, without loss generality, assume that $x_{k+1} \in (F_1 \cap G_i) \setminus (F_1 \cap G_r)$. Note that the exponent of $x_{k+1}$ in $f$ is $w(x_{k+1}, F_1)$, whereas $x_{k+1}^{w(x_{k+1}, F_1) + w(x_{k+1}, G_i)}\mid \x^{\mathbf{a}_1+\mathbf{b}_{i}}$. Therefore, ~$\x^{\mathbf{a}_1+\mathbf{b}_{i}} \nmid f$.

    \smallskip \noindent
    \textsc{Case 3}. We show that $\x^{\mathbf{b}_i+\mathbf{b}_{j}} \nmid f$ for all $1\leq i, j\leq r$. First, suppose that $i = j = r$. Since $x_{p+1}^{2w(x_{p+1},G_r)} \nmid f$, it follows that $\x^{2\mathbf{b}_r} \nmid f$. Secondly, assume that $s \le i \le j \le r-1$ or $1 \le i \le s-1$, $s \le j \le r-1$. Proceeding similarly as in Case 2, there exists $x_{\ell} \in G_j \setminus G_r$ such that $x_{\ell} \notin F_1$. Thus, $x_{\ell}\mid \x^{\mathbf{b}_i+\mathbf{b}_{j}}$, but $x_{\ell}\notin F_1\cup G_r=\supp(f)$.  Hence, it follows that $\mathbf{x}^{\mathbf{b}_i+\mathbf{b}_j} \nmid f$. Next, suppose that $1 \le i \le j \le s-1$. Without loss of generality, assume that $x_{k+1} \in F_1 \cap G_j$. Since $F_1 \cap G_i \supseteq F_1 \cap G_j$, we have $x_{k+1} \in G_i$. By the induction hypothesis, we obtain $w(x_{k+1}, F_1) = w(x_{k+1}, G_i),$ 
    and $w(x_{k+1}, F_1) = w(x_{k+1}, G_j)$. Hence, $w(x_{k+1}, G_i) + w(x_{k+1}, G_j) = 2w(x_{k+1}, F_1)$. Since $x_{k+1} \notin F_1 \cap G_r$, it follows that $x_{k+1}^{w(k+1, G_i) + w(x_{k+1}, G_j)} \nmid f.$ Thus, $\mathbf{x}^{\mathbf{b}_i+\mathbf{b}_j} \nmid f.$
    Finally, suppose that $1 \le i \le s-1$ and $j = r$. If $G_i \not\subseteq F_1 \cup G_r$, then $\x^{\mathbf{b}_i+\mathbf{b}_r} \nmid f$. Assume that $G_i \subseteq F_1 \cup G_r$. Then there exists $x_i \in G_i$ such that $x_i \in G_r \setminus F_1$, and hence $x_i \notin F_1 \cap G_r$. Thus, $x_i^{w(x_i, G_i) + w(x_i, G_r)} \mid \x^{\mathbf{b}_i+\mathbf{b}_r}$,
    while the exponent of $ x_i $ in $ f $ is at most $ w(x_i,G_r) $. Then $x_i^{w(x_i, G_i) + w(x_i, G_r)}  \nmid f$, and consequently, $\x^{\mathbf{b}_i+\mathbf{b}_r} \nmid f$. Hence, $f\notin J^2$, which implies $f \notin I^2$. Thus $f\in I^{(2)}\setminus I^2$, contradicting the assumption that $I$ is Simis. Therefore, $w(x_i,F_1) = w(x_i,G_r)$ for all $x_i \in F_1 \cap G_r$. This completes the proof.
\end{proof}

We conclude this section by studying the Cohen--Macaulay property of weighted monomial ideals associated to grafted simplicial complexes. The Cohen--Macaulay property is one of the fundamental homological properties of monomial ideals, and for simplicial forests it is characterized combinatorially by the notion of \emph{grafting}, introduced by Faridi in \cite{FARIDI2005299}. We investigate the extent to which this characterization extends to the weighted setting. We begin by recalling the definition of a grafted simplicial complex.

\begin{defn}\cite[Definition~7.1]{FARIDI2005299}\label{Grafting_Defn}
    A simplicial complex $\D$ is  a \emph{grafting} of the simplicial complex 
    $\Delta'= \left\langle G_1,\ldots,G_s\right\rangle$
    with the simplices $F_1,\ldots,F_r$, or simply a \emph{grafted} simplicial complex, if
    \[
    \Delta=\left\langle F_1,\ldots,F_r\right\rangle \cup \left\langle G_1,\ldots,G_s\right\rangle
    \]
    and the following conditions hold:
    \begin{enumerate}
        \item $V(\Delta') \subseteq F_1\cup \cdots \cup F_r$,
        \item $F_1,\ldots,F_r$ are all the leaves of $\Delta$,
        \item $\{G_1,\ldots,G_s\}\cap \{F_1,\ldots,F_r\} = \emptyset$,
        \item For $i\neq j$, $F_i\cap F_j= \emptyset$,
        \item if $G_i$ is a joint of $\Delta$, then $\Delta\setminus \left\langle G_i\right\rangle$ is also grafted.
    \end{enumerate}
\end{defn}

Throughout this section, we view a grafted simplicial complex $\D$ as the grafting of the simplicial complex $\D' = \left\langle G_1, \ldots, G_s \right\rangle$ with leaves $F_1, \ldots, F_r$, and write $\D = \D' \cup \la(\D)$, where $\la(\D) = \left\langle F_1, \ldots, F_r \right\rangle$ denotes the subcomplex generated by all the leaves of $\D$. Let $S = \K[x_1, \ldots, x_m, y_1, \ldots, y_M]$, where $V(\D') = \{x_1, \ldots, x_m\}$. For each leaf $F_i, 1 \le i \le r$, define $\Lambda(F_i) := F_i \setminus V(\D')$, the set of free vertices of $F_i$. Then $\bigcup_{i = 1}^{r}\Lambda(F_i) = \left\{y_1, \ldots, y_M\right\}$ is prcisely the set of all free vertices of $\D$.

\medskip
We begin with the following structural lemma. Although this result follows from existing results in the literature, we include a proof for the convenience of the reader, as it will be used repeatedly throughout this section.
\begin{lemma}\label{Vertex_Cover_Lemma}
    Let $\D$ be a grafted simplicial complex, and let $F_i\in\la(\D)$. Then no minimal vertex cover of $\Delta$ contains two distinct vertices of $F_i$.
\end{lemma}
\begin{proof}
    By \cite[Theorem~7.6]{FARIDI2005299}, every minimal vertex cover of $\Delta$ has cardinality $r=|\mathcal{L}(\Delta)|$. Since the leaves of $\mathcal{L}(\Delta)$ are pairwise disjoint, each leaf must be intersected independently by a vertex cover. Hence, if a minimal vertex cover contains two distinct vertices of some leaf $F_i$, then it must also contain at least one vertex from each of the remaining $r-1$ leaves. Consequently, $|\mathcal{C}|\ge 2+(r-1)=r+1$, contradicting the fact that $|\mathcal{C}|=r$.
\end{proof}

Let $I$ be a monomial ideal, and let $\Delta(I)$ denote the associated simplicial complex. Suppose that $\Delta(I)$ is a grafting of a simplicial complex $\Delta'$. We now investigate the Cohen--Macaulay property of $I$. The following lemma plays a key role in the proof of the main theorem and establishes a necessary condition for $I$ to be Cohen--Macaulay.
\begin{lemma}\label{Grafting_Lemma_1}
    Let $I$ be a monomial ideal, and suppose that the associated simplicial complex $\D(I)$ is a grafting of a simplicial complex $\D'$. If either of the following conditions holds, then $I$ has an embedded associated prime.
    \begin{enumerate}[\rm(i)]
        \item There exists $F\in \la(\D(I))$ such that $\alpha_I(F) \ge 2$.
        
        \item There exist $G \in \D'$, $F \in \la(\D(I))$, and integers $1 \le t \le \alpha_I(G)$, $1 \le t' \le \alpha_I(F)$ such that $w(x_i,F,t') < w(x_i,G,t)$ for some vertex $x_i \in F \cap G$.
    \end{enumerate}
\end{lemma}
\begin{proof}
    To prove that $I$ has an embedded associated prime, it suffices to show that $I^{(1)} \ne I$. We do so by constructing a monomial in $I^{(1)}\setminus I$.

    \smallskip
    \noindent
    (i) For each $y_i \in \Lambda(F)$, let $m_i = \min\{w(y_i,F,t) : 1 \le t \le \alpha_I(F)\}$, and consider the monomial
    \[
    f =\prod x_i^{\nu_i(F)}\prod y_i^{m_i}.
    \]
    We show that  $f \in I^{(1)} \setminus I$. Since $\alpha_I(F) \ge 2$, it follows immediately that $f \notin I$. To show $f \in I^{(1)}$, it suffices to show that $f \in Q(P)$ for every $P\in \MinAss(I)$. Let $\Cla$ be the minimal vertex cover corresponding to $P$. By Lemma \ref{Vertex_Cover_Lemma}, if $x\in F\cap \Cla$, then $(F \setminus \{x\}) \cap \Cla = \emptyset$. Hence, whenever $x_i \in F \cap \Cla$, we have $x_i^{\nu_i(F)} \in Q(P)$. Similarly, if $y_i \in \Cla$, then $y_i^{m_i} \in Q(P)$. Therefore, $f \in  Q(P)$ for every $P \in \MinAss(I)$, and consequently, $f \in I^{(1)}$.

    \medskip
    \noindent
    (ii) By part~(i), we may assume that $\alpha_I(F)=1$ for all $F \in \la(\D(I))$. Let $1 \le t \le \alpha_I(G)$ be such that $w(x_i,F,1) < w(x_i,G,t)$, and let $\x^{\mathbf{b}_{t}} \in \G(I)$ be the corresponding generator with $\supp(\x^{\mathbf{b}_{t}}) = G$ and $b_{t,i}=w(x_i, G,t)$.  Now, consider the monomial
    \[
    g = \dfrac{\x^{\mathbf{b}_t}}{x_i}.
    \]
    We show that $g \in I^{(1)} \setminus I$. Since $g\mid \x^{\mathbf{b}_{t}}$, and $g\neq \x^{\mathbf{b}_{t}}$, we have $g\notin I$. It remains to show that $g\in I^{(1)}$. Let $P\in \MinAss(I)$ and $\mathcal{C}$ be the minimal vertex cover of $\D(I)$ corresponding to $P$. By Lemma~\ref{Vertex_Cover_Lemma}, if $x_i \in \Cla$, then $(F \setminus \{x_i\}) \cap \Cla = \emptyset$. Thus, $x_i^{w(x_i,F,1)} \in Q(P)$. On the other hand if $x_i \notin \Cla$, then $\dfrac{\x^{\mathbf{b}_{t}}}{x_i^{b_{t,i}}} \in Q(P)$. Since $w(x_i,F,1)+1 \le b_{t,i}$, write $g =  x_i^{w(x_i,F,1)} \cdot \dfrac{\x^{\mathbf{b}_t}}{x_i^{w(x_i,F, 1)+1}}$. Thus $g\in Q(P)$ for every $P\in \MinAss(I)$, proving that $g\in I^{(1)}$. Therefore, $g\in I^{(1)}\setminus I$.
\end{proof}

The preceding lemma identifies necessary conditions for the Cohen--Macaulayness of $I$. We now establish one of the main results of this section by characterizing the Cohen--Macaulay property for a class of monomial ideals whose associated simplicial complexes are grafted. Since polarization plays a fundamental role in the proof, we first recall this construction using the notation of \cite{peeva2010graded}.
\begin{defn}\cite[Construction~21.7]{peeva2010graded}
\begin{enumerate}
    \item Let $\mathbf{x}^{\mathbf{a}} = x_1^{a_1} \cdots x_n^{a_n}$ be a monomial in $S$. The \emph{polarization} of $\mathbf{x}^{\mathbf{a}}$ is defined to be
    \[
   ( \mathbf{x}^{\mathbf{a}})^{\pol} = (x_1^{a_1})^{\pol} \cdots (x_n^{a_n})^{\pol},
    \]
    where the operator $(*)^{\pol}$ replaces $x_i^{a_i}$ by a product of distinct variables $\prod_{j=1}^{a_i} x_{i,j}$.
    
    \item Let $I = \left(\mathbf{x}^{\mathbf{a}_1}, \dots, \mathbf{x}^{\mathbf{a}_r}\right) \subseteq S$ be a monomial ideal. The \emph{polarization} of $I$ is defined to be the ~ideal
    \[
    I^{\pol} = \left((\mathbf{x}^{\mathbf{a}_1})^{\pol}, \dots, (\mathbf{x}^{\mathbf{a}_r})^{\pol}\right)
    \]
    in a new polynomial ring $S^{\pol} = \mathbb{K}[x_{i,j} \mid 1 \leq i \leq n, 1 \leq j \leq p_i]$, where $p_i$ is the maximum power of $x_i$ appearing in $\mathbf{x}^{\mathbf{a}_1}, \dots, \mathbf{x}^{\mathbf{a}_r}$.
\end{enumerate}

\end{defn}

\begin{thm}\label{Grafting_CM}
    Let $I \subseteq S=\K[x_1, \ldots, x_m,y_1, \ldots,y_M]$ be a monomial ideal satisfying $\mathcal{F}(I)=\mathcal{F}(\sqrt{I})$, and suppose that $\D(I)$ is a grafting of a simplicial complex $\D'$. Let $J = I \cap \K[x_1, \ldots, x_m]$, and assume that $J$ admits a standard linear weighting. Then $S/I$ is Cohen--Macaulay if and only if the following conditions hold: 
    \begin{enumerate}[\rm(i)]
        \item $\alpha_I(F)=1$ for every $F\in \la(\D(I))$;
        \item for every $F\in \la(\D(I))$, $G\in \D'$, and every vertex $x_i\in F\cap G$, $w(x_i,F)\ge w(x_i,G)$.
    \end{enumerate}
\end{thm}
\begin{proof}
    Suppose that $S/I$ is Cohen--Macaulay. Since the ideal $I$ has no embedded primes, by Lemma~\ref{Grafting_Lemma_1} we have $\alpha_I(F)=1$ for every $F\in\mathcal{L}(\Delta(I))$,  and $w(x_i,F)\ge w(x_i,G)$ for every $F\in\mathcal{L}(\Delta(I))$, $G\in\Delta'$, and every vertex $x_i\in F\cap G$. Hence conditions {\rm(i)} and {\rm(ii)} hold.

    Conversely, assume that conditions {\rm(i)} and {\rm(ii)} are satisfied. Write $\D' = \left\langle G_1, \ldots, G_s\right\rangle$ and let $\mathcal{L}(\Delta(I))=\{ F_1, \ldots, F_r\}$. Without loss of generality, assume that $G_i$ is the joint of $F_i$ for each $1\leq i\leq r$. Then $F_i=\{x_{i_1}, \ldots, x_{i_{p_i}}\} \cup \{y_{i_1}, \ldots, y_{i_{q_i}}\}$,  where $F_i\cap G_i= \{x_{i_1}, \ldots, x_{i_{p_i}}\}$, and $y_{i_1}, \ldots, y_{i_{q_i}} \in \Lambda(F_i)$.

    Let $R= \K[u_1, \ldots, u_r]$, and let $K \subseteq R$ be the monomial ideal obtained from $I$ by replacing every vertex of the leaf $F_i$ with the variable $u_i$, for each $1 \le i \le r$. Then $u_i^{M_i}\in K$ for all $1\leq i\leq r$, where $M_i=\sum_{z \in F_i }w(z,F_i)$. Consequently, the ideal $K\subseteq R$ is Artinian. We now consider the polarizarions of the rings $S$ and $R$, and define a map
    \[
    \phi :  S^{\pol} \longrightarrow R^{\pol}
    \] as follows. For each $1\leq i\leq r$, set $T_{i,0} = 0$ and define
    $T_{i,c} = T_{i,c-1}+w(x_{i_c}, F_i)= \sum_{h=1}^{c} w(x_{i_h}, F_i)$ for $1\leq c\leq p_i$. Further, set $T_{i, p_i+1}= T_{i,p_i}+w(y_{i_1}, F_i)$ and $ T_{i, p_i+d}=T_{i,p_i+d-1}+ w(y_{i_d}, F_i)= T_{i,p_i} +\sum_{h=1}^{d}w(y_{i_h}, F_i) $ for $2\leq d \leq q_i$. We now define the $\K$-algebra homomorphism $\phi$ by $\phi(x_{i_j, \ell})= u_{i, \ell + T_{i,{j-1}}}$ for all $1\leq \ell \leq w(x_{i_j}, F_i)$ and $1\leq i\leq r$. The map is well defined since, by condition {\rm(ii)},  $w(x_i,G)\leq w(x_i,F)$ for every $F\in\mathcal{L}(\Delta(I))$, $G\in\Delta'$, and every vertex $x\in F\cap G$. Moreover, $\phi$ is bijective. Moreover, map $\phi$ is  bijective. We claim that $ \phi\big(I^{\pol}\big)= K^{\pol}$. Indeed, by construction, there is a one-to-one correspondence between the monomial generators corresponding to the leaves $F_i$ and the monomials $\left(u_i^{M_i}\right)^{\pol}$. On the other hand, since the vertices of $\D'$ have standard linear weights, and the condition {\rm(ii)} holds, the construction of the map $\phi$ yields a one-to one correspondence between monomial generators corresponding to the facets $G\in \D'$ and the monomials in $K^{\pol}$ that are not pure powers of variables. Hence, $\phi(I^{\pol})=K^{\pol}$, and therefore $\phi$ induces an isomorphism of $\K$-algebras 
    \[
    \phi :  S^{\pol}/I^{\pol} \longrightarrow R^{\pol}/K^\pol .
    \]
    Since $R/K$ is Artinian, it is Cohen--Macaulay. By \cite[Corollary~1.6.3]{herzog2011monomial}, it follows that the ring $R^{\pol}/K^{\pol}$ is Cohen--Macaulay. Since $S^{\pol}/I^{\pol}\cong R^{\pol}/K^{\pol}$, it follows that $S^{\pol}/I^{\pol}$ is Cohen--Macaulay. Consequently, $S/I$ is Cohen--Macaulay.
\end{proof}

The following example shows that the hypothesis that $J$ admits a standard linear weighting is not necessary for $I$ to be Cohen--Macaulay.

\begin{exmp}\label{eg:4CM}
Consider the ideal
\[
    I=\l x_1^2x_2x_3x_4x_5x_6,\, x_1^3x_2x_7x_8,\, x_4x_5x_7x_8,\, x_1^3x_2x_3y_1,\, x_4x_5y_2,\, x_6y_3,\, x_7x_8y_4 \r 
\]
in $\K[x_1,\ldots,x_8,y_1,\ldots,y_4]$. Then $J=\l x_1^2x_2x_3x_4x_5x_6,\, x_1^3x_2x_7x_8,\,
x_4x_5x_7x_8 \r$. Since the weight assigned to the vertex $x_1$ is not uniform, the ideal $J$ does not admit a standard linear weighting. Nevertheless, a computation in Macaulay2 shows that $I$ is Cohen--Macaulay.
\end{exmp}

\section{Support-3 Monomial Ideals}\label{sec:supp3}

In this section, we investigate the Mendez--Pinto--Villarreal conjecture for support-$3$ monomial ideals. A monomial ideal whose minimal generators are of the form $x_i^a x_j^b x_k^c$, where $i,j,k$ are distinct and $a,b,c\geq 1$, is called a support-$3$ monomial ideal. If $I\subseteq S$ is a support-$3$ monomial ideal, then the associated simplicial complex $\Delta(I)$ is a pure simplicial complex of dimension $2$. The Simis property and the packing property have been studied for certain classes of cubic square-free monomial ideals arising from simple graphs \cite{alilooee2021packing}. However, a complete characterization of the Simis property for cubic square-free monomial ideals remains an open problem.

In \cite[Theorem 2.3]{Bordoloi2026}, the authors established the MPV conjecture for the class of support-$2$ monomial ideals. A key ingredient in their proof is that the minimal generators corresponding to a fixed support admit a natural total ordering. For support-$3$ monomial ideals, however, it is unclear whether such a natural ordering exists, even when one considers the total degrees of the generators. More importantly, it is not evident how such an ordering, if it exists, could be exploited to study the Simis property. Instead of pursuing this approach, we compare the supports of the minimal generators directly and analyze the corresponding monomials through their exponent vectors. We begin by introducing the notation and conventions for support-$3$ monomial ideals that will be used throughout the remainder of this section.

\begin{setup}\label{Set-up:Sup_3}
    Let $I\subseteq S$ be a support-$3$ monomial ideal with minimal irreducible primary decomposition, and suppose that $\alpha_I(F)=1$ for every $F\in\mathcal{F}(I)$. Let $F_1,F_2\in\mathcal{F}(I)$ satisfy $|F_1\cap F_2|=2$. After relabeling the variables, we may assume that
    \[
    F_1=\{x_2,x_3,x_4\},
    \qquad
    F_2=\{x_1,x_3,x_4\}.
    \]
    We focus on the vertices $x_3,x_4\in F_1\cap F_2$, and our goal is to prove that they have uniform weights in all generators whose support is contained in $\{x_1,\ldots,x_4\}$. To this end, let $\mathbf{x}^{\mathbf{a}_1},\mathbf{x}^{\mathbf{a}_2}\in\mathcal{G}(I)$ be such that $\supp(\mathbf{x}^{\mathbf{a}_1})=F_1, \supp(\mathbf{x}^{\mathbf{a}_2})=F_2$, and assume that $w(x_3,F_1)<w(x_3,F_2)$. Let
    \[
    J=I\cap \K[x_1,x_2,x_3,x_4].
    \]
    Then $J$ is a support-$3$ monomial ideal in $\K[x_1,x_2,x_3,x_4]$ containing $\mathbf{x}^{\mathbf{a}_1}$ and $\mathbf{x}^{\mathbf{a}_2}$. Since there are exactly four $3$-element subsets of $\{x_1,x_2,x_3,x_4\}$, every minimal generator of $J$ has support $F_i:=\{x_1,x_2,x_3,x_4\}\setminus\{x_i\}, 1\le i\le 4$.
    Accordingly, we have 
    \[
    (\mathbf{x}^{\mathbf{a}_1},\mathbf{x}^{\mathbf{a}_2})\subseteq J\subseteq (\mathbf{x}^{\mathbf{a}_1},\mathbf{x}^{\mathbf{a}_2},
    \mathbf{x}^{\mathbf{a}_3},\mathbf{x}^{\mathbf{a}_4}).
    \]
    If $J=(\mathbf{x}^{\mathbf{a}_1},\mathbf{x}^{\mathbf{a}_2})$, then Proposition~\ref{Prop_Sim_tree} yields
    \[
        w(x_3,F_1)=w(x_3,F_2)
        \quad\text{and}\quad
        w(x_4,F_1)=w(x_4,F_2),
    \]
    and hence the desired conclusion follows. Thus, it remains to consider the cases where $\mathbf{x}^{\mathbf{a}_3}\in J$, $\mathbf{x}^{\mathbf{a}_4}\in J$, or both. Rather than treating these cases separately, we assume that $J=(\mathbf{x}^{\mathbf{a}_1},\mathbf{x}^{\mathbf{a}_2}, \mathbf{x}^{\mathbf{a}_3},\mathbf{x}^{\mathbf{a}_4})$, that is, both $\mathbf{x}^{\mathbf{a}_3}$ and $\mathbf{x}^{\mathbf{a}_4}$ belong to $J$. The remaining cases follow by the same arguments after omitting the generator corresponding to the missing support.
\end{setup}

\par
Next, we establish a number of technical lemmas and propositions needed for the proof of the main theorem. The proofs of these results follow a common strategy. Under suitable hypotheses on the supports, weights, and exponents of the minimal generators, we construct a monomial that lies in the second symbolic power but not in the second ordinary power. Consequently, the ideal is not Simis in degree $2$. Since this argument requires repeated verification of membership in both the symbolic and ordinary powers, we first record several elementary observations that simplify these computations.

\begin{rmk}\label{rmk:minexp}
    Let $I\subseteq S$ be a support-$3$ monomial ideal, and let $f$ be a monomial. We follow the notations and assumptions in Set-up~\ref{Set-up:Sup_3}. Suppose that $\supp(f)\subseteq\{x_1,\ldots,x_4\}$. Then, for every $s\ge1$, $f\in I^s \Leftrightarrow f\in J^s.$ Consequently, to determine whether $f\in I^s$, it is enough to work with the ideal $J$.
    
    Now fix $1\le i\le4$, and write $\{x_1,x_2,x_3,x_4\}\setminus\{x_i\}=\{x_j,x_k,x_\ell\}$. Following Set-up~\ref{Set-up:Sup_3}, let $\mathbf{x}^{\mathbf{a}_i}\in \G(I)$ be such that $\supp(\mathbf{x}^{\mathbf{a}_i})=\{x_1,x_2,x_3,x_4\}\setminus\{x_i\}$. Assume that
    \[
    w(x_i,F_j)\le w(x_i,F_k),\qquad w(x_i,F_j)\le w(x_i,F_\ell),
    \]
    and that the exponent of $x_i$ in $f$ is precisely $w(x_i,F_j)$; that is, $x_i^{w(x_i,F_j)}\mid f, x_i^{w(x_i,F_j)+1}\nmid f$. Comparing the exponent of $x_i$ in the generators of $J$, we conclude that, in order to prove $f\notin J^2$, it suffices to verify that $f\notin \left( \mathbf{x}^{2\mathbf{a}_i}, \mathbf{x}^{\mathbf{a}_i+\mathbf{a}_j}, \mathbf{x}^{\mathbf{a}_i+\mathbf{a}_k}, \mathbf{x}^{\mathbf{a}_i+\mathbf{a}_\ell} \right)$.
\end{rmk}

\begin{rmk}\label{rmk:symbolic}
    For $1\le i\le n$, set $M_i=\max\{w(x_i,F):F\in\mathcal F(I)\}$. Suppose that $x_i^{sM_i}\mid f$ for some monomial $f\in S$.
    Let $P\in\MinAss(I)$, and $Q(P)$ denotes the irreducible primary component corresponding to $P$. If $x_i\in P$, then $x_i^{M_i}\in Q(P)$, and therefore $f\in Q(P)^s$. Consequently, to prove that $f\in I^{(s)}$, it suffices to verify that $f\in Q(P)^s$ for those minimal primes $P\in\MinAss(I)$ with $x_i\notin P$.
\end{rmk}

Throughout the next three lemmas, we adopt the notation and assumptions of Set-up~\ref{Set-up:Sup_3}. We focus on the variables $x_1$ and $x_2$, together with the supports $F_1$ and $F_2$, and investigate the consequences when the standard linear weight condition fails for these variables. The next three lemmas provide a case-by-case analysis of this situation.

\begin{lemma}\label{Prop:Supp_3:Min}
     With the notation and assumptions of Set-up~\ref{Set-up:Sup_3}, suppose that  
     \[
     w(x_1, F_2) \le \min \left\{w(x_1, F_3), w(x_1, F_4) \right\}~~ \text{or}  ~~w(x_2, F_1) \le \min \left\{w(x_2, F_3), w(x_2, F_4) \right\}.
     \]
     Then $I^{(2)} \ne I^2$.
\end{lemma}

\begin{proof}
    Suppose that $w(x_1, F_2) \le \min \left\{w(x_1, F_3), w(x_1, F_4) \right\}$. Consider the monomial
    \[
        f = x_1^{w(x_1,F_2)}x_2^{w(x_2,F_1)}x_3^{\max\left\{2w(x_3,F_1),w(x_3,F_2) \right\}}x_4^{2M_4}.
    \]
    We claim that $f \in I^{(2)} \setminus I^2$. We first show that $f \notin I^2$. By Remark~\ref{rmk:minexp}, it suffices to verify $f \notin \l \mathbf{x}^{2\mathbf{a}_1}, \mathbf{x}^{\mathbf{a}_1+\mathbf{a}_2},\mathbf{x}^{\mathbf{a}_1+\mathbf{a}_3}, \mathbf{x}^{\mathbf{a}_1+\mathbf{a}_4}\r$. We obtain the desired non-divisibilities from
    \begin{align*}
        w(x_2,F_1)&<2w(x_2,F_1),\\
        \max\{2w(x_3,F_1),w(x_3,F_2)\}&<w(x_3,F_1)+w(x_3,F_2),\\
        w(x_2,F_1)&<w(x_2,F_1)+w(x_2,F_3),\\
        w(x_2,F_1)&<w(x_2,F_1)+w(x_2,F_4).
    \end{align*}
    Here, the second inequality follows from the assumption $w(x_3,F_1)<w(x_3,F_2)$ in Set-up~\ref{Set-up:Sup_3}.
    Next, we show that $f \in I^{(2)}$. As $x_4^{2M_4} \mid f$, by Remark~\ref{rmk:symbolic} it is enough to show that $f \in Q(P)^2$ for those $P \in \MinAss(I)$ such that $x_4 \notin P$. Hence, we consider the following possibilities.
    If $x_{3}^{w(x_3,F_1)} \in Q(P)$, then $x_{3}^{2w(x_3,F_1)} \in Q(P)^2$. Since $x_{3}^{2w(x_3,F_1)}\mid f_1$, we have $f_1 \in Q(P)^2$.
    Now, suppose that $x_3^{w(x_3,F_1)} \notin Q(P)$.
    Since the irreducible decomposition of $I$ is minimal, this implies that $x_2^{w(x_2,F_1)} \in Q(P)$. 
    Furthermore, if $x_{1}^{w(x_1,F_2)}\in Q(P)$, then $x_{1}^{w(x_1,F_2)}x_2^{w(x_2,F_1)} \in Q(P)^2$.
    As $x_{1}^{w(x_1,F_2)}x_2^{w(x_2,F_1)} \mid f$, it implies that $f \in Q(P)^2$.
    On the other hand, if $x_{1}^{w(x_1,F_2)}\notin Q(P)$, then $x_{3}^{w(x_3,F_2)}\in Q(P)$ as the irreducible decomposition of $I$ is minimal.
    Thus, we have $x_{2}^{w(x_2,F_1)}x_3^{w(x_3,F_2)} \in Q(P)^2$. Hence, $f \in Q(P)^2$ and therefore $f \in I^{(2)}$. The proof for the case $w(x_2, F_1) \le \min \left\{w(x_2, F_3), w(x_2, F_4)\right\}$ is analogous and is therefore ~omitted.
\end{proof}

\begin{lemma} \label{Prop:Supp_3:Intermediate}
     With the notation and assumptions of Set-up~\ref{Set-up:Sup_3}, suppose that one of the following conditions holds: 
    \begin{enumerate}[\rm(i)]
        \item $w(x_1,F_3)< w(x_1,F_2) \le w(x_1,F_4)$ \text{and} $w(x_2,F_4) < w(x_2,F_1) \le w(x_2,F_3)$.
        \item $\max\{w(x_1,F_3) , w(x_1,F_4)\} < w(x_1,F_2)~ \text{and} ~w(x_2,F_4)< w(x_2,F_1) \le w(x_2,F_3)$.
        \item $ w(x_1,F_4) < w(x_1,F_2) \le w(x_1,F_3) ~\text{and} ~w(x_2,F_4) < w(x_2,F_1) \le w(x_2,F_3)$.
    \end{enumerate}
    Then $I^{(2)} \ne I^2$.
\end{lemma}
\begin{proof}
    In each of the given cases, we shall construct a suitable monomial and make repeated use of the Remark~\ref{rmk:minexp} and Remark~\ref{rmk:symbolic} and to conclude the assertion.
    
    \smallskip
    \noindent (i) Consider the monomial,
    \[
    f_1 = x_1^{w(x_1,F_2)}x_2^{w(x_2,F_1)}x_3^{\max\left\{2w(x_3,F_1),w(x_3,F_2) \right\}}x_4^{2M_4}.
    \]
    We claim that $f_1 \in I^{(2)} \setminus I^2$. We first show that $f_1 \notin I^2$. Note that for $i \in \{1,3,4\}$, $w(x_2,F_i) \ge 1$. Hence, it follows that $f_1 \notin \l \mathbf{x}^{2\mathbf{a}_1}, \mathbf{x}^{\mathbf{a}_1+\mathbf{a}_3} ,\mathbf{x}^{\mathbf{a}_1+\mathbf{a}_4} \r$. Again, for $i \in \{2,3,4\}$, $w(x_1,F_i) \ge 1$. Thus, we obtain $f_1 \notin \l \mathbf{x}^{2\mathbf{a}_2}, \mathbf{x}^{\mathbf{a}_2+\mathbf{a}_3},\mathbf{x}^{\mathbf{a}_2+\mathbf{a}_4} \r$. By the assumption $w(x_3,F_1)<w(x_3,F_2)$ in Set-up~\ref{Set-up:Sup_3}, we obtain $\max\{2w(x_3, F_1),w(x_3,F_2)\} < w(x_3, F_1) + w(x_3, F_2)$. Thus, $\mathbf{x}^{\mathbf{a}_1+\mathbf{a}_2} \nmid f_1$. For the remaining generators observe that
    \begin{align*}
        w(x_2,F_1)&<2w(x_2,F_3),\\
        w(x_2,F_1)&<w(x_2,F_3)+w(x_2,F_4),\\
        w(x_1,F_2)&<2w(x_1,F_4).
    \end{align*}
    These inequalities yield, respectively, $f_1 \notin \l\mathbf{x}^{2\mathbf{a}_3},\mathbf{x}^{\mathbf{a}_3+\mathbf{a}_4}, \mathbf{x}^{2\mathbf{a}_4}\r$. Hence, $f_1 \notin I^2$. Arguing exactly as in the proof that $f\in I^{(2)}$ in Lemma~\ref{Prop:Supp_3:Min}, we obtain $f_1\in I^{(2)}$. Therefore $f_1 \in I^{(2)} \setminus I^2$.
    
    \medskip
    
    \noindent (ii)
    For this case, consider the monomial 
    \[
    f_2 = x_1^{w(x_1,F_4)}x_2^{\max\left\{2w(x_2,F_4),w(x_2,F_1)\right\}}x_3^{2M_3}x_4^{2M_4}.
    \]
    Firstly, we show that $f_2 \notin I^2$. Note that for $i \in \{2,3,4\}$, $w(x_1,F_i) \ge 1$. Thus, it follows that $f_2 \notin \l\mathbf{x}^{\mathbf{a}_2+\mathbf{a}_4},\mathbf{x}^{\mathbf{a}_3+\mathbf{a}_4},\mathbf{x}^{2\mathbf{a}_4} \r$. Next, since $w(x_1,F_4) < w(x_1,F_2)$, we obtain $f_2 \notin \l \mathbf{x}^{\mathbf{a}_1+\mathbf{a}_2}, \mathbf{x}^{2\mathbf{a}_2},\mathbf{x}^{\mathbf{a}_2+\mathbf{a}_3} \r$. It remains to show that $f_2 \notin \l\mathbf{x}^{2\mathbf{a}_1}, \mathbf{x}^{\mathbf{a}_1+\mathbf{a}_4}, \mathbf{x}^{\mathbf{a}_1+\mathbf{a}_3}, \mathbf{x}^{2\mathbf{a}_3}\r$. However, the required non-divisibilities follow from the following exponent comparisons:
    \begin{align*}
        2w(x_2,F_4)&<2w(x_2,F_1),\\
        2w(x_2,F_4)&<w(x_2,F_1)+w(x_2,F_4),\\
        2w(x_2,F_4)&<w(x_2,F_1)+w(x_2,F_3),\\
        \max\{2w(x_2,F_4),w(x_2,F_1)\} &<2w(x_2,F_3).
    \end{align*} 
    Next, we show that $f_2\in I^{(2)}$. Let $P\in\MinAss(I)$ with $x_3,x_4\notin P$. First, assume that $x_2^{w(x_2,F_4)} \in Q(P)$. Then it is evident that $f_2 \in Q(P)^2$ as $x_2^{2w(x_2,F_4)} \mid f_2$. Next, if $x_2^{w(x_2,F_4)} \notin Q(P)$, since the irreducible decomposition of $I$ is minimal, it follows that $x_1^{w(x_1,F_4)} \in Q(P)$. Note that we also have $x_2^{w(x_2,F_1)} \in Q(P)$ as $x_3, x_4 \notin P$ and $F_1 = \{x_2,x_3,x_4\}$. Consequently, $x_1^{w(x_1,F_4)}x_2^{w(x_2,F_1)} \in Q(P)^2$ and hence $f_2 \in Q(P)^2$.  By using Remark~\ref{rmk:symbolic}, we conclude that $f_2 \in Q(P)^2$ for every $P \in \MinAss(I)$. This proves that $f_2 \in I^{(2)}$.
    
    \medskip
    \noindent (iii) We construct the monomial
    \[
    f_3 = x_1^{\max\left\{2w(x_1,F_4),w(x_1,F_2)\right\}}x_2^{w(x_2,F_4)}x_3^{2M_3}x_4^{2M_4},
    \]
    in order to show that $I^{(2)}\neq I^2$.
    By Remark~\ref{rmk:minexp}, to show $f_3 \notin I^2$, it is enough to verify that $f_3 \notin \l \mathbf{x}^{\mathbf{a}_1 + \mathbf{a}_2},\mathbf{x}^{2\mathbf{a}_2},\mathbf{x}^{\mathbf{a}_2+\mathbf{a}_3}, \mathbf{x}^{\mathbf{a}_2+\mathbf{a}_4} \r$. 
    Consider the following exponent comparisons:
    \begin{align*}
        \max\{2w(x_1,F_4),w(x_1,F_2)\}&<2w(x_1,F_2),\\
        2w(x_1,F_4)&<w(x_1,F_2)+w(x_1,F_4),\\
        w(x_2,F_4)&<\min\{w(x_2,F_1),w(x_2,F_3)\}.
    \end{align*}
    The first two inequalities yield $\mathbf{x}^{2\mathbf{a}_2}\nmid f_3$ and $\mathbf{x}^{\mathbf{a}_2+\mathbf{a}_4}\nmid f_3$ respectively, while the last inequality shows that $f_3 \notin \l \mathbf{x}^{\mathbf{a}_1+\mathbf{a}_2},  \mathbf{x}^{\mathbf{a}_2+\mathbf{a}_3}\r$. Hence, $f_3\notin I^2$. We next verify that $f_3\in I^{(2)}$. Firstly, assume that $P\in\MinAss(I)$ with $x_3,x_4\notin P$. We consider the following cases. If $x_1^{w(x_1,F_4)} \in Q(P)$, then $x_1^{2w(x_1,F_4)} \in Q(P)^2$, and since     $x_1^{2w(x_1,F_4)} \mid f_3$ we obtain $f_3 \in Q(P)^2$. Otherwise, $x_1^{w(x_1,F_4)}\notin Q(P)$ and the minimality of the irreducible decomposition of $I$ implies that $x_2^{w(x_2,F_4)} \in Q(P)$. Moreover since $x_3,x_4\notin P$ and $F_2=\{x_1,x_3,x_4\}$, we have $x_1^{w(x_1,F_2)}\in Q(P)$. Consequently,  $x_2^{w(x_2,F_4)}x_1^{w(x_1,F_2)} \in Q(P)^2$. As $x_2^{w(x_2,F_4)}x_1^{w(x_1,F_2)}\mid f_3$, we again obtain $f_3\in Q(P)^2$. Next, as $x_3^{2M_3}x_4^{2M_4} \mid f_3$, using Remark~\ref{rmk:symbolic}, we obtain  $f_3\in Q(P)^2$ for every $P \in \MinAss(I)$. Thus, $f_3\in I^{(2)}$.               
\end{proof}

\begin{lemma} \label{Prop:Supp_3:Max}
    With the notation and assumptions of of Set-up~\ref{Set-up:Sup_3}, suppose that 
    \[
    \max\{w(x_1,F_3),w(x_1,F_4)\} < w(x_1,F_2) \quad  \text{and} \quad  \max\{w(x_2,F_3), w(x_2,F_4) \}< w(x_2,F_1),
    \]
    then $I^{(2)} \ne I^2$.
\end{lemma}
\begin{proof}
    The proof proceeds by a case analysis. The cases below are mutually exclusive and collectively exhaustive. In each case, we exhibit a monomial $f\in I^{(2)}\setminus I^2$.

    \noindent
    \textsc{Case 1.} Assume that $ w(x_1,F_3) \le w(x_1,F_4)$ and $ w(x_2,F_3) \le w(x_2,F_4)$. We consider the ~monomial 
    \[
    f_1 = x_1^{w(x_1,F_3)}x_2^{\max\{2w(x_2,F_3),w(x_2,F_1)\}}x_3^{2M_3}x_4^{2M_4}.
    \]
    In view of Remark~\ref{rmk:minexp}, to show $f_1 \notin I^2$, we verify that $f_1 \notin \l \mathbf{x}^{2\mathbf{a}_1}, \mathbf{x}^{\mathbf{a}_1+\mathbf{a}_2}, \mathbf{x}^{\mathbf{a}_1+\mathbf{a}_3}, \mathbf{x}^{\mathbf{a}_1+\mathbf{a}_4}\r$. Indeed, the required non-divisibilities follow from the corresponding exponent comparisons:
    \begin{align*}
        2w(x_2,F_3) &< 2w(x_2,F_1),\\
        w(x_1,F_3) &< w(x_1,F_2),\\
        \max\{2w(x_2,F_3),w(x_2,F_1)\}&<w(x_2,F_1)+w(x_2,F_3),\\
        2w(x_2,F_3)&<w(x_2,F_1)+w(x_2,F_4).
    \end{align*}
    To show that $f_1\in I^{(2)}$, we prove that $f_1\in Q(P)^2$ for every $P\in\MinAss(I)$, where $Q(P)$ denotes the primary component corresponding to $P$. Since $x_3^{2M_3}x_4^{2M_4}\mid f_1$, by Remark~\ref{rmk:symbolic} it is enough to show $f \in Q(P)^2$ for those minimal primes $P$ such that $x_3,x_4\notin P$. We proceed with the following cases. If $x_2^{w(x_2,F_3)} \in Q(P)$, then $x_2^{2w(x_2,F_3)} \in Q(P)^2$. As $x_2^{2w(x_2,F_3)} \mid f_1$, we have $f_1 \in Q(P)^2$. If $x_2^{w(x_2,F_3)} \notin Q(P)$, then since the irreducible decomposition of $I$ is minimal and $x_4\notin P$, we must have $x_1^{w(x_1,F_3)} \in Q(P)$. Furthermore, as $x_3, x_4 \notin P$ and $F_1 = \{x_2,x_3,x_4\}$, we already know that $x_2^{w(x_2,F_1)} \in Q(P)$. Hence, we have $x_1^{w(x_1,F_3)} x_2^{w(x_2,F_1)} \in Q(P)^2$ as $x_1^{w(x_1,F_3)} x_2^{w(x_2,F_1)} \mid f_1$. This gives $f_1 \in Q(P)^2$. Therefore, $f_1 \in I^{(2)} \setminus I^2$. 

    \smallskip
    \noindent
    \textsc{Case 2.} Assume that $ w(x_1,F_4) < w(x_1,F_3)$ and $ w(x_2,F_3) < w(x_2,F_4)$. We further consider two subcases.

    \noindent
    \textsc{Subcase 2A.} Assume that $2w(x_2,F_3) <  w(x_2,F_1)$. We consider the monomial \[
    f_{2A}= x_1^{\max\{2w(x_1,F_4),w(x_1,F_3)\}}x_2^{\max\{2w(x_2,F_3),w(x_2,F_4)\}}x_3^{2M_3}x_4^{2M_4}.\]
    Firstly, we show that $f_{2A} \notin I^2$. Note that since $\max\{2w(x_2,F_3),\,w(x_2,F_4)\} < w(x_2,F_1)$ we obtain $f_{2A} \notin \l \mathbf{x}^{2\mathbf{a}_1}, \mathbf{x}^{\mathbf{a}_1+\mathbf{a}_2}, \mathbf{x}^{\mathbf{a}_1+\mathbf{a}_3}, \mathbf{x}^{\mathbf{a}_1+\mathbf{a}_4}\r$. For the remaining generators of $J^2$, we have the following exponent comparisons:
    \begin{align*}
        \max\{2w(x_1,F_4),w(x_1,F_3)\} &<2w(x_1,F_2),\\
        2w(x_1,F_4)&<2w(x_1,F_3),\\
        \max\{2w(x_2,F_3),w(x_2,F_4)\}&<2w(x_2,F_4),\\
        2w(x_1,F_4)&<w(x_1,F_2)+w(x_1,F_3),\\
        \max\{2w(x_1,F_4),w(x_1,F_3)\}&<w(x_1,F_2)+w(x_1,F_4),\\
        2w(x_2,F_3)&<w(x_2,F_3)+w(x_2,F_4).
    \end{align*}
    Thus $f_{2A} \notin \l \mathbf{x}^{2\mathbf{a}_2}, \mathbf{x}^{2\mathbf{a}_3}, \mathbf{x}^{2\mathbf{a}_4},\mathbf{x}^{\mathbf{a}_2+\mathbf{a}_3},\mathbf{x}^{\mathbf{a}_2+\mathbf{a}_4}, \mathbf{x}^{\mathbf{a}_3 + \mathbf{a}_4} \r$. To prove that $f_{2A}\in I^{(2)}$, by Remark~\ref{rmk:symbolic} it suffices to verify that $f_{2A}\in Q(P)^2$ for those $P\in\MinAss(I)$ such that $x_3,x_4\notin P$. Now, if $x_2^{w(x_2,F_3)} \in Q(P)$, then $x_2^{2w(x_2,F_3)} \in Q(P)^2$, and hence $f_{2A} \in Q(P)^2$. Likewise, if $x_1^{w(x_1,F_4)}\in Q(P)$, then $x_1^{2w(x_1,F_4)}\in Q(P)^2$, which again implies that $f_{2A}\in Q(P)^2$. Finally, suppose that $x_2^{w(x_2,F_3)}, x_1^{w(x_1,F_4)}  \notin Q(P)$. Then the minimality of the irreducible decomposition of $I$ forces  $x_1^{w(x_1,F_3)}, x_2^{w(x_2,F_4)} \in Q(P)$. Therefore, $x_1^{w(x_1,F_3)}x_2^{w(x_2,F_4)} \in Q(P)^2$, and hence $f_{2A} \in Q(P)^2$. 

    \noindent
    \textsc{Subcase 2B.} Assume that $w(x_2,F_1) \le 2w(x_2,F_3)$. In this case, we consider the monomial,
    \[
    f_{2B}= x_1^{w(x_1, F_4)}x_2^{2w(x_2,F_4)}x_3^{2M_3}x_4^{2M_4}.
    \]
    From Remark~\ref{rmk:minexp}, to show $f_{2B} \notin I^2$ it suffices to show that $f_{2B} \notin  \l \mathbf{x}^{2\mathbf{a}_1}, \mathbf{x}^{\mathbf{a}_1+\mathbf{a}_2}, \mathbf{x}^{\mathbf{a}_1+\mathbf{a}_3}, \mathbf{x}^{\mathbf{a}_1+\mathbf{a}_4}\r$. This follows from the following exponent comparisons:
    \begin{align*}
        2w(x_2,F_4)&<2w(x_2,F_1),\\
        w(x_1,F_4)&<\min\{w(x_1,F_2),w(x_1,F_3)\},\\
        2w(x_2,F_4)&<w(x_2,F_1)+w(x_2,F_4).
    \end{align*}
    Next, we show that $f_{2B}\in Q(P)^2$ for every $P\in\MinAss(I)$. By Remark~\ref{rmk:symbolic} it is enough to show $f \in Q(P)^2$ for those $P\in\MinAss(I)$ such that $x_3,x_4\notin P$. Firstly, if $x_2^{w(x_2,F_4)} \in Q(P)$, then $x_2^{2w(x_2,F_4)} \in Q(P)^2$. Hence $f_{2B} \in Q(P)^2$. If $x_2^{w(x_2,F_4)} \notin Q(P)$, then by minimality of the irreducible primary decomposition of $I$, we have $x_1^{w(x_1,F_4)} \in Q(P)$. Moreover, since $x_3, x_4 \notin P$ and $F_1 = \{x_2,x_3,x_4\}$, we have $x_2^{w(x_2,F_1)} \in Q(P)$. Thus, $x_1^{w(x_1,F_4)}x_2^{w(x_2,F_1)} \in Q(P)^2$. Since $w(x_2,F_1) \le 2w(x_2,F_3) < 2w(x_2,F_4)$, it follows that $x_1^{w(x_1,F_4)}x_2^{w(x_2,F_1)} \mid f_{2B}$. Therefore,  $f_{2B} \in Q(P)^2$ and hence $f_{2B} \in I^{(2)}$.  

    \smallskip
    \noindent
    \textsc{Case 3.} Assume that $w(x_1,F_3)\leq w(x_1,F_4)$ and $w(x_2,F_4)<w(x_2,F_3)$. If $w(x_1,F_3)=w(x_1,F_4)$, then, by interchanging the roles of $F_3$ and $F_4$ in the Case~$1$ and choosing the corresponding monomial, we obtain $I^{(2)}\neq I^2$. If $w(x_1,F_3)<w(x_1,F_4)$, the same conclusion follows by interchanging $F_3$ and $F_4$ in Case~$2$ and applying an analogous argument.

    \smallskip
    \noindent
    \textsc{Case 4.} Assume that $w(x_1,F_4)<w(x_1,F_3)$ and $w(x_2,F_4)\leq w(x_2,F_3)$. Both possibilities, namely $w(x_2,F_4)=w(x_2,F_3)$ and $w(x_2,F_4)<w(x_2,F_3)$, follow from Case~$1$ by interchanging $F_3$ and $F_4$. Hence, $I^{(2)}\neq I^2$.
\end{proof}

We are now in a position to prove the main result of Set-up~\ref{Set-up:Sup_3}. We continue to use the notation introduced there.
\begin{prop}\label{Lem:Support3Int_2}
    Let $I$ be a support-3 monomial ideal such that
    \begin{enumerate}
        \item the irreducible decomposition of $I$ is minimal;
        \item there exist $F_1,F_2\in \F(I)$ such that $|F_1\cap F_2|=2$;
        \item there is $x_i\in F_1\cap F_2$ such that $w(x_i,F_1)\neq w(x_i,F_1)$.
    \end{enumerate}
    Then $I^{(2)}\neq I^2$.
\end{prop}
\begin{proof}
    By Lemma~\ref{lem:Support_gen}, it suffices to consider the case where $\alpha_I(F)= 1$ for every $F \in \F(I)$. We adopt the notation of Set-up~\ref{Set-up:Sup_3}. Let $\mathbf{x}^{\mathbf{a}_1},\mathbf{x}^{\mathbf{a}_2}\in \mathcal{G}(I)$ with $\supp(\mathbf{x}^{\mathbf{a}_1})=F_1= \{x_2,x_3,x_4\}, \supp(\mathbf{x}^{\mathbf{a}_2})=F_2= \{x_1,x_3,x_4\}$, and set $J=I\cap \K[x_1,x_2,x_3, x_4]$. There are four possible choices of the ideal $J$, namely
    \begin{enumerate}
        \item $J =  \left(\mathbf{x}^{\mathbf{a}_1},\mathbf{x}^{\mathbf{a}_2} \right )$,
        \item $J= \left(\mathbf{x}^{\mathbf{a}_1},\mathbf{x}^{\mathbf{a}_2}, \mathbf{x}^{\mathbf{a}_3} \right )$,

        \item $J= \left(\mathbf{x}^{\mathbf{a}_1},\mathbf{x}^{\mathbf{a}_2}, \mathbf{x}^{\mathbf{a}_4} \right )$,

        \item $J= \left(\mathbf{x}^{\mathbf{a}_1},\mathbf{x}^{\mathbf{a}_2}, \mathbf{x}^{\mathbf{a}_3}, \mathbf{x}^{\mathbf{a}_4} \right )$.
    \end{enumerate}
    Case $(1)$ has already been treated in Set-up~\ref{Set-up:Sup_3}. In case $(4)$, the support of $J$ is the complete $3$-uniform hypergraph on four vertices. We now establish this case. The proofs of cases $(2)$ and $(3)$ are entirely analogous and are obtained by omitting the weights corresponding to the minimal generators that do not belong to the minimal generating set of $J$.
    
    Since $x_1\in F_i$ for each $i\in\{2,3,4\}$, we distinguish cases according to the order relations among $w(x_1,F_2)$, $w(x_1,F_3)$, and $w(x_1,F_4)$ considered in the preceding three lemmas.

    \noindent
    \textsc{Case 1.} Suppose that $w(x_1,F_2)\leq \min\{w(x_1,F_3),w(x_1,F_4)\}$. Then Lemma~\ref{Prop:Supp_3:Min} implies that ~$I^{(2)}\neq I^2$. 

    \smallskip
    \noindent
    \textsc{Case 2.} Suppose that $w(x_1,F_2)>\max\{w(x_1,F_3),w(x_1,F_4)\}$.
    We further distinguish cases according to the position of $w(x_2,F_1)$ with respect to $w(x_2,F_3)$ and $w(x_2,F_4)$. 
    
    \smallskip
    \noindent\textit{(i)} If $w(x_2,F_1)\le \min\{w(x_2,F_3),w(x_2,F_4)\}$, then the conclusion follows from Lemma~\ref{Prop:Supp_3:Min}.
    
    \smallskip
    \noindent\textit{(ii)} If $w(x_2,F_1)>\max\{w(x_2,F_3),w(x_2,F_4)\}$, then the conclusion follows from Lemma~\ref{Prop:Supp_3:Max}.
    
    \smallskip
    \noindent\textit{(iii)} If $w(x_2,F_4)<w(x_2,F_1)\le w(x_2,F_3)$, then Lemma~\ref{Prop:Supp_3:Intermediate}(ii) yields $I^{(2)}\neq I^2$.
    
    \smallskip
    \noindent\textit{(iv)} If $w(x_2,F_3)<w(x_2,F_1)\le w(x_2,F_4)$, then the conclusion follows from Lemma~\ref{Prop:Supp_3:Intermediate}(ii) after interchanging the roles of $F_3$ and $F_4$.

    \smallskip
    \noindent
    \textsc{Case 3.} Suppose that $\min\{w(x_1,F_3),w(x_1,F_4)\}< w(x_1,F_2) \le \max\{w(x_1,F_3),w(x_1,F_4)\}$. First, assume that $w(x_1,F_4)\leq w(x_1,F_3)$. Then $w(x_1,F_4)<w(x_1,F_2)\leq w(x_1,F_3)$. We now consider the same four possibilities as in Case~2, according to the relative position of $w(x_2,F_1)$ with respect to $w(x_2,F_3)$ and $w(x_2,F_4)$.

    \smallskip
    \noindent\textit{(i)} If $w(x_2,F_1)\le \min\{w(x_2,F_3),w(x_2,F_4)\},$
        then the conclusion follows from Lemma~\ref{Prop:Supp_3:Min}.
        
    \smallskip
    \noindent\textit{(ii)} If $w(x_2,F_1)> \max\{w(x_2,F_3),w(x_2,F_4)\}$, then the conclusion follows by an argument analogous to that of Lemma~\ref{Prop:Supp_3:Intermediate}(ii).
        
    \smallskip
    \noindent\textit{(iii)} If $w(x_2,F_4)<w(x_2,F_1)\le w(x_2,F_3)$, then Lemma~\ref{Prop:Supp_3:Intermediate}(iii) yields $I^{(2)}\neq I^2$.
        
    \smallskip
    \noindent\textit{(iv)} Lastly, if $w(x_2,F_3)<w(x_2,F_1)\le w(x_2,F_4)$, then the conclusion follows from Lemma~\ref{Prop:Supp_3:Intermediate}(i) after interchanging the roles of $F_3$ and $F_4$.

    Finally, assume that $w(x_1,F_3)\leq w(x_1,F_4)$. Then $w(x_1,F_3)<w(x_1,F_2)\leq w(x_1,F_4)$. Proceeding exactly as above, we obtain $I^{(2)}\neq I^2$, completing the proof.
\end{proof}

\begin{setup}\label{Set-up:Sup_3(1)}
    Let $I\subseteq S$ be a support-$3$ monomial ideal with minimal irreducible primary decomposition, and suppose that $\alpha_I(F)=1$ for every $F\in\mathcal{F}(I)$. Let $F_1,F_2\in\mathcal{F}(I)$ satisfy $|F_1\cap F_2|=1$. After relabeling the variables, assume that
    \[
        F_1=\{x_1,x_2,x_3\}, \qquad F_2=\{x_1,x_4,x_5\},
    \]
    and let $\mathbf{x}^{\mathbf a_1},\mathbf{x}^{\mathbf a_2}\in\mathcal G(I)$ satisfy $\supp(\mathbf{x}^{\mathbf a_1})=F_1,\supp(\mathbf{x}^{\mathbf a_2})=F_2$, with $w(x_1,F_1)<w(x_1,F_2)$. Set $J=I\cap \K[x_1,\ldots,x_5]$. Then $(\mathbf{x}^{\mathbf a_1},\mathbf{x}^{\mathbf a_2}) \subseteq J\subseteq\widetilde J$, where $\widetilde J=(\mathbf{x}^{\mathbf a_1},\ldots,\mathbf{x}^{\mathbf a_{10}})$ is the support-$3$ monomial ideal whose support is the complete $3$-uniform hypergraph on $\{x_1,\ldots,x_5\}$. We denote the supports of $\mathbf{x}^{\mathbf a_3},\ldots,\mathbf{x}^{\mathbf a_{10}}$ by
    \[
    \begin{aligned}
        F_3&=\{x_2,x_3,x_4\}, &
        F_4&=\{x_2,x_3,x_5\},\\
        F_5&=\{x_2,x_4,x_5\}, &
        F_6&=\{x_3,x_4,x_5\},\\
        F_7&=\{x_1,x_2,x_4\}, &
        F_8&=\{x_1,x_2,x_5\},\\
        F_9&=\{x_1,x_3,x_4\}, &
        F_{10}&=\{x_1,x_3,x_5\}.
    \end{aligned}
    \]
\end{setup}

\begin{prop}\label{lem:supp3_int1}
    Let $I$ be a support-3 monomial ideal such that
    \begin{enumerate}
        \item the irreducible decomposition of $I$ is minimal;
        \item there exist $F_1,F_2\in \F(I)$ such that $|F_1\cap F_2|=1$, and $w(x_i,F_1)\neq w(x_i,F_1)$ where $F_1\cap F_2=\{x_i\}$. 
    \end{enumerate}
    Then $I^{(2)}\neq I^2$.
\end{prop}
\begin{proof}
    By Lemma~\ref{lem:Support_gen}, it suffices to consider the case where $\alpha_I(F)=1$ for every $F \in \F(I)$. To prove this, we adopt the notations in Set-up~\ref{Set-up:Sup_3(1)}. Let $F_1=\{x_1,x_2,x_3\}, F_2=\{x_1,x_4,x_5\}$ and $w(x_1,F_1) < w(x_1,F_2)$. Let $\mathbf{x}^{\mathbf{a}_1},\mathbf{x}^{\mathbf{a}_2}\in \mathcal{G}(I)$ be such that $\supp(\mathbf{x}^{\mathbf{a}_1})=F_1, \supp(\mathbf{x}^{\mathbf{a}_2})=F_2$.  Consider the ideal $J=I\cap \K[x_1,\dots , x_5]$. 

    We first eliminate the generators whose supports contain $x_1$. Observe that, for each $7\le i\le 10$, we have $|F_1\cap F_i|=2 \text{ and } x_1\in F_1\cap F_i.$ Suppose that $\mathbf{x}^{\mathbf a_{\ell}}\in J$ for some $7\le \ell \le 10$. If $w(x_1,F_1)\neq w(x_1,F_{\ell})$, then Proposition~\ref{Lem:Support3Int_2} yields $I^{(2)}\neq I^2$. Hence, we may assume that $w(x_1,F_1)=w(x_1,F_{\ell})$. Since $|F_2\cap F_{\ell}|=2, x_1\in F_2\cap F_{\ell},$ and $w(x_1,F_{\ell})=w(x_1,F_1)<w(x_1,F_2)$, it follows that $w(x_1,F_{\ell})\neq w(x_1,F_2)$. Applying Proposition~\ref{Lem:Support3Int_2} once again, we obtain $I^{(2)}\neq I^2$. Consequently, we may assume that $\mathbf{x}^{\mathbf a_i}\notin J\text{ for all }7\le i\le 10$.
    
    We consider the monomial 
    \[
    f = x_1^{\max\left\{2w(x_1,F_1),w(x_1,F_2)\right\}}x_2^{w(x_2,F_1)}x_3^{w(x_3,F_1)}x_4^{w(x_4,F_2)}x_5^{w(x_5,F_2)}, 
    \]
    and claim that $f\in I^{(2)}\setminus I^2$. We first show that $f\notin I^2$. Since $f\in K[x_1,\ldots,x_5]$, it is enough to prove that $f\notin J^2$. We consider the following cases.

    \smallskip
    \noindent
    \textsc{Case 1}. We show that $\mathbf{x}^{\mathbf{a}_1+\mathbf{a}_i}\nmid f$ for all $1\leq i\leq 6$. If $i=2$, then ${w(x_1,F_1)+w(x_1,F_2)}> \max\left\{2w(x_1,F_1),w(x_1,F_2)\right\}$, and hence $\mathbf{x}^{\mathbf{a}_1+\mathbf{a}_2}\nmid f$. Now, let $i\neq 2$. Since $F_i\cap \{x_2,x_3\}\neq \emptyset$, either $x_2\in F_i$ or $x_3\in F_i$. If $x_2\in F_i$, then $w(x_2,F_1)+w(x_2, F_i)>w(x_2, F_1)$, and so $\mathbf{x}^{\mathbf{a}_1+\mathbf{a}_i}\nmid f$. The argument is identical when $x_3\in F_i$.

    \smallskip
    \noindent
    \textsc{Case 2}. We show that $\mathbf{x}^{\mathbf{a}_2+\mathbf{a}_i}\nmid f$ for all $1\leq i\leq 6$. The case $i=1$ was settled in Case 1. For $i\neq 1$, we have $F_i\cap \{x_4, x_5\}\neq \emptyset$, and the same argument as in Case~1 shows that $\mathbf{x}^{\mathbf a_2+\mathbf a_i}\nmid f$.

    \smallskip
    \noindent
    \textsc{Case 3}. We show that $\mathbf{x}^{\mathbf{a}_i+\mathbf{a}_j}\nmid f$ for all $3\leq i,j\leq 6$. Assume that $\mathbf{x}^{\mathbf{a}_3}\in J$. Observe that $F_1\cap F_3=\{x_2,x_3\}$. If $w(x_2,F_1)\neq w(x_2,F_3)$ or $w(x_3,F_1)\neq w(x_3,F_3)$, then Proposition~\ref{Lem:Support3Int_2} implies that $I^{(2)}\neq I^2$, as desired. Hence, we may assume that $w(x_2,F_1)=w(x_2,F_3) \text{ and } w(x_3,F_1)=w(x_3,F_3)$. Since the exponent of $x_2$ in $f$ is $w(x_2,F_1)$, it follows that $\mathbf{x}^{2\mathbf{a}_3},\, \mathbf{x}^{\mathbf{a}_3+\mathbf{a}_4},\, \mathbf{x}^{\mathbf{a}_3+\mathbf{a}_5} \nmid f.$ Similarly, since the exponent of $x_3$ in $f$ is $w(x_3,F_1)$, we have $\mathbf{x}^{\mathbf{a}_3+\mathbf{a}_6}\nmid f$. The proofs for the remaining cases, namely when $\mathbf{x}^{\mathbf{a}_i}\in J$ for $4\le i\le 6$, are analogous and are therefore ~omitted.

    Now, it remains to show that $f\in I^{(2)}$. Let $P\in\MinAss(I)$, and let $Q(P)$ denote the corresponding primary component. We consider the following cases.

\smallskip
    \noindent
    \textsc{Case 1}. Suppose that $x_1^{w(x_1,F_1)}\in Q(P)$. Then $x_1^{2w(x_1,F_1)}\in Q(P)^2$. Since $x_1^{2w(x_1,F_1)}\mid f$, we obtain $f\in Q(P)^2$.

\smallskip
    \noindent
    \textsc{Case 2}. Suppose that $x_1^{w(x_1,F_1)}\notin Q(P)$. By the minimality of the irreducible primary decomposition of $I$, either $x_2^{w(x_2,F_1)}\in Q(P)$ or $x_3^{w(x_3,F_1)}\in Q(P)$. 
    
    \smallskip
    \noindent
    \textsc{Subcase 2A}. If $x_2^{w(x_2,F_1)}\in Q(P)$, then, again by minimality of the irreducible primary decomposition of $I$, at least one of $x_1^{w(x_1,F_2)}, x_4^{w(x_4,F_2)}, x_5^{w(x_5,F_2)}$ belongs to $Q(P)$. Consequently, one of $x_2^{w(x_2,F_1)}x_1^{w(x_1,F_2)}, x_2^{w(x_2,F_1)}x_4^{w(x_4,F_2)}, x_2^{w(x_2,F_1)}x_5^{w(x_5,F_2)}$ belongs to $Q(P)^2$. Since each of these monomials divides $f$, we conclude that $f\in Q(P)^2$.
    
    \smallskip
    \noindent
    \textsc{Subcase 2B}. If $x_3^{w(x_3,F_1)}\in Q(P)$, then proceeding similarly as in Subcase~2A, one can show that $f\in Q(P)^2$.

    \smallskip
    \noindent
    Hence, in every case, $f\in Q(P)^2$. Since $P$ was arbitrary, $f\in I^{(2)}$, and this completes the ~proof.
\end{proof}

We conclude this section with the following theorem, which establishes Conjecture~\ref{conj} for support-3 monomial ideals.
\begin{thm}\label{support-3}
    Let $I$ be a support-$3$ monomial ideal. Then, Conjecture \ref{conj} is true.
\end{thm}

\begin{proof}
    By Lemma~\ref{lem:Support_gen}, we have $\alpha_I(F)=1$ for every $F\in\mathcal{F}(I)$. Hence, Lemmas~\ref{Lem:Support3Int_2} and \ref{lem:supp3_int1} imply that every variable has standard linear weights. Therefore, there exist a square-free monomial ideal $J$ and a standard linear weighting $w$ such that $I=J_w$. Since $I$ is Simis, \cite[Corollary~5.5(b)]{mendez2024symbolic} implies that $J$ is Simis.
\end{proof}

Unlike the case of edge ideals of simple graphs, the Simis property for cubic square-free monomial ideals remains open in full generality and appears to be considerably more difficult to characterize. Nevertheless, for the class of $3$-path ideals of graphs, both the packing property and a complete combinatorial characterization of the Simis property have been established \cite{alilooee2021packing}. Consequently, these results, together with Theorem~\ref{support-3}, yield a complete combinatorial characterization of monomial ideals $I$ with minimal irreducible decomposition such that $\sqrt{I}$ is the $3$-path ideal of a graph.

\bibliography{ref.bib}
\bibliographystyle{plain}

\vspace{0.5cm}
\end{document}